\documentclass[11pt]{amsart}
\usepackage{verbatim,amssymb,amsmath,graphicx,hyperref}
  \scrollmode

\newcommand{\R}{{\mathbb R}}
\newcommand{\N}{{\mathbb N}}

\newcommand{\EE}{{\mathbb E}}
\newcommand{\PP}{{\mathbb P}}

\newcommand{\sgn}{\operatorname{sgn}}
\newcommand{\eps}{\varepsilon}
\newcommand{\mmu}{\mu}

\newcommand{\F}{{\mathcal F}}

\theoremstyle{plain}
\newtheorem{theorem}{Theorem}
\newtheorem{prop}{Proposition}
\newtheorem{lemma}{Lemma}

\theoremstyle{definition}

\begin{document}

\title[Lower error bounds for strong approximation of SDEs with discontinuous drift]{Sharp lower error bounds for strong approximation of SDEs with discontinuous drift coefficient by coupling of noise}

\author[M\"uller-Gronbach]
{Thomas M\"uller-Gronbach}
\address{
Faculty of Computer Science and Mathematics\\
University of Passau\\
Innstrasse 33 \\
94032 Passau\\
Germany} \email{thomas.mueller-gronbach@uni-passau.de}

\author[Yaroslavtseva]
{Larisa Yaroslavtseva}
\address{
Faculty of Mathematics and Economics \\
University of Ulm\\
Helmholzstrasse 18 \\
89069 Ulm\\
Germany} \email{larisa.yaroslavtseva@uni-ulm.de}

\begin{abstract}
In the past decade, an intensive study of strong approximation of stochastic differential equations (SDEs) with  a drift coefficient that has discontinuities in space has begun. In the majority of these results it is assumed that the drift coefficient satisfies piecewise regularity conditions and that the diffusion coefficient  is globally Lipschitz continuous and non-degenerate at the discontinuities of the drift coefficient. Under this type of assumptions the best $L_p$-error rate obtained so far for approximation of scalar 
SDEs at the final time  is $3/4$ in terms of the number of evaluations of the  
driving Brownian motion. 
In the present article we  
prove for the first time in the literature sharp lower error bounds
for such SDEs. 
We show that for a huge class of 
additive noise driven
SDEs of this type the $L_p$-error rate  $3/4$ can not be improved.

For the proof of this result we employ a novel technique
by studying equations with coupled noise:
we reduce the analysis of the $L_p$-error of an arbitrary approximation based on evaluation of the driving Brownian motion at finitely many times  to the analysis of the $L_p$-distance of two solutions of the same equation that are driven by Brownian motions that are coupled at the given time-points and independent, conditioned on their values at these points.
To obtain lower bounds for the latter quantity, we prove a new quantitative version of positive association for bivariate normal 
random variables $(Y,Z)$ by providing explict lower bounds for the covariance $\text{Cov}(f(Y),g(Z))$ in case  of piecewise Lipschitz continuous functions $f$ and $g$. 
In addition it turns out that our proof technique also leads to  lower error bounds for estimating occupation time functionals $\int_0^1 f(W_t)\, dt$ of a Brownian motion $W$, which substantially extends known results for the case of $f$ being an indicator function. 
\end{abstract}
\maketitle

\section{Introduction and main results}\label{s3}

Consider a scalar autonomous stochastic differential equation (SDE)
\begin{equation}\label{sde0}
\begin{aligned}
dX_t & = \mu(X_t) \, dt + \sigma(X_t) \, dW_t, \quad t\in [0,1],\\
X_0 & = x_0
\end{aligned}
\end{equation}
with deterministic initial value $x_0\in\R$,  drift coefficient $\mu\colon\R\to\R$,  diffusion coefficient $\sigma\colon \R\to\R$ and $1$-dimensional 
standard
Brownian motion $W$. Assume that the SDE \eqref{sde0} has a unique strong solution $X$. In this article we  study 
$L_p$-approximation of $X_1$ by means of methods that  use finitely many evaluations of the driving Brownian motion $W$  in the case when the drift coefficient $\mu$  may have discontinuity points.

SDEs with a  drift coefficient that has discontinuities in space arise e.g. in mathematical finance, insurance and stochastic control problems.
In the past decade, an intensive study of strong approximation of such SDEs has begun. All investigations carried out so far study the performance of classical numerical methods for such equations or present new numerical methods and provide corresponding upper error bounds. See~\cite{ g98b, gk96b} for results on convergence in probability and  almost sure convergence  of the Euler-Maruyama scheme  and~\cite{DG18, GLN17, HalidiasKloeden2008,  LS16,  LS15b, LS18, MGY19b, MGY20, NS19, NSS19,  Tag16, Tag2017b, Tag2017a, PS19}
for results on $L_p$-approximation. In the present article we provide for the first time lower error bounds that are valid for any approximation of $X_1$ based on a finite number of evaluations of $W$ and are sharp for a huge class of 
additive noise driven
SDEs of this type.

To be more precise, consider the following conditions on the coefficients $\mu$ and $\sigma$. 
\begin{itemize}
\item[($\mu$1)] There exist $k\in\N$ 
and 
$-\infty=\xi_0<\xi_1<\ldots < \xi_k <\xi_{k+1}=\infty$ such that
 $\mu$ is Lipschitz continuous on  $(\xi_{i-1}, \xi_i)$ for every $i\in\{1, \ldots, k+1\}$, \vspace{0.1cm}
\item[($\sigma$1)] $\sigma$ is Lipschitz continuous on $\R$ and $\sigma(\xi_i) \neq 0$ for every $i\in\{1,\ldots,k\}$. \vspace{0.1cm}
\end{itemize}

If $\mu$ and $\sigma$ satisfy ($\mu$1) and ($\sigma$1), respectively, then 
the SDE~\eqref{sde0} has a unique strong solution $X$, see~\cite{LS16}.  $L_p$-approximation of $X_1$ under the assumptions 
 ($\mu$1) and ($\sigma$1)
has been studied in~\cite{ LS16,  LS15b, LS18, MGY20,  NSS19}. 
In particular, in~\cite{LS16, LS15b} the first numerical 
method
has been  
constructed  which achieves, under ($\mu$1) and ($\sigma$1), an $L_2$-error rate of at least $1/2$  in terms of the number of evaluations of $W$.   
This method is based on a suitable transformation of the strong solution $X$ into a strong solution
of an SDE with Lipschitz continuous coefficients. Thereafter,
 in~\cite{NSS19} an adaptive Euler-Maruyama scheme
has been
constructed, which achieves, under ($\mu$1) and ($\sigma$1), 
an $L_2$-error rate of at least $1/2-$ in terms of the average number of evaluations of $W$ used by this method. Finally, in~\cite{MGY20} it
has been 
proven that, 
under ($\mu$1) and ($\sigma$1), the standard  Euler-Maruyama 
scheme with $n$ equidistant steps
 in fact
achieves for all $p\in [1,\infty)$ an $L_p$-error rate of at least $1/2$  
in terms of the number $n$ of evaluations of $W$
as in the classical case of SDEs with 
globally Lipschitz continuous coefficients.

Recently in~\cite{MGY19b} the first higher-order method has been constructed for such SDEs, which
 achieves for all $p\in[1, \infty)$ an $L_p$-error rate $3/4$ if $\mu$ and $\sigma$ satisfy ($\mu$1) and ($\sigma$1) and additionally the following  piecewise regularity assumptions \vspace{0.1cm}
\begin{itemize}
\item[($\mu$2)]  $\mu$ has a Lipschitz continuous derivative on   $(\xi_{i-1}, \xi_i)$  for every $i\in\{1, \ldots, k+1\}$,\vspace{0.1cm}
\item[($\sigma$2)]  $\sigma$ has a Lipschitz continuous derivative on  $(\xi_{i-1}, \xi_i)$  for every $i\in\{1, \ldots, k+1\}$. \vspace{0.1cm}
\end{itemize}
More precisely, in~\cite{MGY19b} the following theorem has been proven.
\begin{theorem}\label{Thm0}
Assume that $\mu$ satisfies ($\mu$1) and ($\mu$2) and that $\sigma$ satisfies ($\sigma$1) and ($\sigma$2). Then there exist a sequence of measurable functions $g_n\colon \R^n\to \R$, $n\in\N$, such that for every $p\in [1,\infty)$
 there exists $c\in (0,\infty)$ such that for every $n\in\N$,
\[
\EE\bigl[|X_1-g_n(W_{1/n},W_{2/n},\dots,W_1)|^p\bigr]^{1/p} \le c/n^{3/4}.
\]
\end{theorem}

  The approximations $g_n(W_{1/n},W_{2/n},\dots,W_1)$ in Theorem~\ref{Thm0} are obtained by applying a suitable transformation $G\colon\R\to\R$ to the strong solution $X$ of the SDE \eqref{sde0} such that the transformed solution $Y=(G(X_t))_{t\in[0,1]}$ is a strong solution of a new SDE with sufficiently regular coefficients. A Milstein-type scheme $\widehat Y_n$ with $n$ equidistant steps is then used to approximate $Y_1$ and $G^{-1}(\widehat Y_n)$ yields an approximation of $X_1$, which satisfies the upper error bound in Theorem~\ref{Thm0}. See~\cite[Section 4]{MGY19b} for details. We add that in~\cite{NS19} it has been proven that if  $\sigma=1$ and
$\mu$ is  bounded and piecewise $C^2_b$ then the standard Euler-Maruyama scheme in fact achieves an $L_2$-error rate of  at least $3/4-$ in terms of the number of evaluations of $W$. Note that in the latter case the Euler-Maruayama scheme coincides with the Milstein scheme.

It is well known that in the classical case of globally Lipschitz continuous coefficients $\mu$ and $\sigma$, the Milstein scheme achieves  
      for all $p\in[1, \infty)$ an $L_p$-error rate of at least $1$ 
in terms of the number of evaluations of $W$ under the additional regularity assumption that $\mu$ and $\sigma$ have  bounded and Lipschitz continuous derivatives,
see e.g.~\cite{HMGR01}. 
 It is therefore natural to ask    whether 
 there exists a method based on finitely many evaluations of $W$
 that achieves under the assumptions ($\mu$1), ($\mu$2) and ($\sigma$1), ($\sigma$2) a better $L_p$ error rate than the rate $3/4$ guaranteed by Theorem~\ref{Thm0}.
  To the best of our knowledge the answer to this question
was not known  in the literature up to now. In the present article we answer this question in the negative. More precisely, we show that no numerical method based on $n$ evaluations of $W$
can achieve an $L_p$-error rate better than $3/4$ in terms of $n$ if $\sigma=1$ and $\mu$ satisfies, additionally to ($\mu$1) and ($\mu$2), the conditions
 \begin{itemize}
\item[($\mu$3)]  $\exists i\in\{1,\dots,k\}\colon\, \mu(\xi_i+)\not=\mmu(\xi_i-)$, \vspace{0.1cm}
\item[($\mu$4)]  $\mu$ is increasing,\vspace{0.1cm}
\item[($\mu$5)]  $\mu$ is bounded. \vspace{0.1cm}
\end{itemize}
More formally, the main result of this article is the following theorem.

\begin{theorem}\label{Thm1}
Assume that $\mu$ satisfies ($\mu$1) to ($\mu$5) and that $\sigma=1$. Then 
 there exists $c\in(0, \infty)$ such that for all $n\in\N$, 
\begin{equation}\label{l3}
\inf_{\substack{
       t_1,\dots ,t_n \in [0,1]\\
        g \colon \R^n \to \R \text{ measurable} \\
       }}	 \EE\bigl[|X_1-g(W_{t_1}, \ldots, W_{t_n})|\bigr]\geq \frac{c}{n^{3/4}}. 
\end{equation}
\end{theorem}

Note that Theorem~\ref{Thm1} also shows that the $L_p$-error rate $3/4$ can in general not be improved even then when additionally to the assumptions ($\mu$1), ($\mu$2) and ($\sigma$1), ($\sigma$2) further piecewise regularity assumptions are imposed on $\mu$ and $\sigma$. Not even the property of being piecewise infinitely often differentiable with uniformly bounded derivatives may help.

As an example consider the strong solution $X$ of~\eqref{sde0} with $\sigma = 1$ and $\mu=1_{[0,\infty)}$. We then have by Theorem~\ref{Thm0} and  Theorem~\ref{Thm1} that for all 
$n\in\N$,
\[
\frac{c_1}{n^{3/4}}\le 
\inf_{\substack{
       t_1,\dots ,t_n \in [0,1]\\
        g \colon \R^n \to \R \text{ measurable} \\
       }}	 \EE\Bigl[\Bigl|\int_0^11_{[0,\infty)}(X_s)\, ds-g(W_{t_1}, \ldots, W_{t_n})\Bigr|^p\Bigr]^{1/p} \le \frac{c_2}{n^{3/4}},
\]
where $c_1,c_2\in (0,\infty)$ depend only on 
$p$.

We briefly discuss the additional conditions ($\mu$3) to ($\mu$5) used in Theorem~\ref{Thm1}. First note that property ($\mu$1) implies that the limits $\mu(\xi_i-) = \lim_{x\uparrow \xi_i}\mu(x)$ and $\mu(\xi_i+) = \lim_{x\downarrow \xi_i}\mu(x)$ exist for all $i\in\{1,\dots,k\}$, see Lemma~\ref{basics}.
In the presence of ($\mu$1) and ($\mu$2), the condition  ($\mu$3) can not be waived in Theorem~\ref{Thm1}: if $\mu$ satisfies ($\mu$1) and ($\mu$2) and $\mu(\xi_i+)=\mu(\xi_i-)$ for every $i\in\{1, \ldots, k\}$
then $\mu_{|_{\R\setminus\{\xi_1, \ldots, \xi_k\}}}$ has a Lipschitz continuous extension
 $\tilde\mu\colon\R\to\R$, 
  which has a Lipschitz continuous derivative on   $(\xi_{i-1}, \xi_i)$  for every $i\in\{1, \ldots, k+1\}$,                       
 and $X$ is also a strong solution of the SDE \eqref{sde0} with  $\mu$ replaced by $\tilde \mu$ and  $\sigma = 1$. By~\cite[Theorem 2]{MGY19b} it then follows that the Milstein scheme achieves at least an $L_p$-error rate $1$ for approximation of $X_1$, which is in contradiction to the lower bound~\eqref{l3}. The condition ($\mu$4) is of major importance for our  proof of Theorem~\ref{Thm1}, see the discussion of our proof strategy below, and it is unclear to us whether this condition could be weakened or even dropped.
  With respect to condition ($\mu$5) we believe that its use in the proof of Theorem~\ref{Thm1} could be avoided by fully exploiting the fact that under the condition ($\mu$1) the drift coefficient $\mu$ satisfies a linear growth condition, see Lemma~\ref{basics}, which in turn implies that $X_1$ has finite moments of any order.  

We add that lower error bounds for strong approximation of scalar SDEs at a single time are already  provided in~\cite{hhmg2019} and~\cite{m04}, but in the setting of Theorem~\ref{Thm1} these bounds turn out to be much too small. In fact, if $\sigma=1$, $\mu$ satisfies ($\mu$1) and if there exists an open interval $I\subset \R$ and a time $t_0\in[0,1)$ such that $\mu$ is three times continuously differentiable on $I$, $\mu'\neq 0$ on $I$ and $\PP(X_{t_0}\in I)>0$ then~\cite[Theorem 6]{hhmg2019} implies only that~\eqref{l3} holds with $c/n^{3/4}$ 
replaced by $c/n$. Note, however, that the  lower bound $c/n$ in~\cite[Theorem 6]{hhmg2019} is also valid for approximations of $X_1$ that may use $n$ sequential evaluations of $W$ on average, while Theorem~\ref{Thm1} only covers approximations that are based on evaluation of $W$ at $n$ fixed discretization sites.
We conjecture that the lower bound in Theorem~\ref{Thm1} does not hold anymore if one allows for sequential evaluation of $W$ and that 
one can achieve under the assumptions ($\mu$1), ($\mu$2) and ($\sigma$1), ($\sigma$2) an $L_p$-error rate $1$ by a method based on adaptive step-size control. The proof of this conjecture will be the subject of future work.

We turn to a sketch of our proof strategy for Theorem~\ref{Thm1}, which heavily differs from the  techniques known from the literature that have been employed so far for establishing lower error bounds in the context of approximation of SDEs. To avoid technical details we restrict to the analysis of the $L_2$-error and we only study approximations of $X_1$ that are based on equidistant evaluations of $W$. Fix  $n\geq 2$  and put $t_i=i/n$ for $i\in\{0,1,\dots,n\}$.  

The central  idea 
of our proof
is to consider a second Brownian motion $\widetilde W$ such that  $W$ and $\widetilde W$ are coupled at the points $t_1,\dots,t_n$ but independent, conditioned on $W_{t_1},\dots,W_{t_n}$, and to study the mean square distance of the two corresponding strong solutions $X$ and $\widetilde X$ of the SDE~\eqref{sde0} at time $1$. Formally, let $\overline W$ denote the piecewise linear interpolation of $W$ at the points $t_0,\dots,t_n$, let $B=W-\overline W$ denote the corresponding piecewise Brownian bridge process, and define
\[
\widetilde W = \overline W + \widetilde B,
\]
where $\PP^B = \PP^{\widetilde B}$ and 
 $W,\widetilde B$ 
are independent. Then for all $t\in[0,1]$,
\[
X_t = x_0 +\int_0^t\mu(X_s)\, ds + W_t,\quad \widetilde X_t = x_0 +\int_0^t\mu(\widetilde X_s)\, ds + \widetilde W_t,
\]
and for every measurable function $g\colon\R^n\to\R$ one has
\begin{equation}\label{i0}
\EE\bigl[|X_1-g(W_{t_1},\dots,W_{t_n})|^2\bigr]^{1/2}\ge \frac{1}{2}\,\EE\bigl[|X_1-\widetilde X_1|^2\bigr]^{1/2},
\end{equation}
see Lemma~\ref{lemmanew02}.
 By the coupling of $W$ and $\widetilde W$ we have
\[
X_{t_i}-\widetilde X_{t_i} = X_{t_{i-1}}-\widetilde X_{t_{i-1}} + \int_{t_{i-1}}^{t_i} (\mu(X_s)-\mu(\widetilde X_s))\, ds
\]
for all $i\in\{1,\dots,n\}$, which yields
\begin{equation}\label{i1}
\begin{aligned}
\EE\bigl[|X_1-\widetilde X_1|^2\bigr] & = 2\sum_{i=1}^n \underbrace{\EE\Bigl[(X_{t_{i-1}}-\widetilde X_{t_{i-1}})\, \int_{t_{i-1}}^{t_i} (\mu(X_s)-\mu(\widetilde X_s))\, ds\Bigr]}_{=:\,m_i} \\
& \qquad\qquad + \sum_{i=1}^n \underbrace{\EE\Bigl[\Bigr(\int_{t_{i-1}}^{t_i} (\mu(X_s)-\mu(\widetilde X_s))\, ds\Bigr)^2\Bigr]}_{=:\,d_i}.
\end{aligned}
\end{equation}

Using the assumption that $\mu$ is increasing and the fact that pathwise uniqueness holds for equation~\eqref{sde0} we obtain by a comparison theorem for SDEs that 
\begin{equation}\label{i2}
m_i\ge 0
\end{equation}
 for all $i\in\{1,\dots,n\}$, see Lemma~\ref{mixed}.
 
 For the analysis of the terms  $d_i$ we first show that for all 
$i\ge n/2+1$ 
 the solutions $X$ and $\widetilde X$ may be replaced on
  $[t_{i-1},t_i]$ 
  by the processes
\[
(X_{t_{i-1}} + W_s-W_{t_{i-1}})_{ s\in[t_{i-1},t_i]}\,\text{ and } \,( X_{t_{i-1}} + \widetilde W_s-\widetilde W_{t_{i-1}})_{ s\in[t_{i-1},t_i]},
\] 
respectively, in the sense that
\begin{equation}\label{i3xx}
d_i \ge \frac{1}{4}
\EE\Bigl[\underbrace{\Bigr(\int_{t_{i-1}}^{t_i} \bigl(\mu(X_{t_{i-1}} + W_s-W_{t_{i-1}})-\mu( X_{t_{i-1}} + \widetilde W_s-\widetilde W_{t_{i-1}})\bigr)\, ds\Bigr)^2}_{R_i}\Bigr] -\frac{c}{n^{5/2+1/16}},
\end{equation}
see Lemma~\ref{diagonal1}.
 To obtain~\eqref{i3xx} we establish 
 appropriate $L_p$-estimates for the differences $X_{t_{i-1}}-\widetilde X_{t_{i-1}}$,
 see Lemma~\ref{lab2}, 
 and 
  $L_2$-estimates
 for the total time of $(X_s)_{s\in [t_{i-1},t_i]}$ and $(X_{t_{i-1}} + W_s-W_{t_{i-1}})_{s\in [t_{i-1},t_i]}$ lying on different sides of a fixed horizontal line
 in order to cope with the discontinuities of the drift coefficient $\mu$, see Lemma~\ref{Yproc}.

It remains to provide lower bounds for the terms $\EE[R_i]$ for $i\ge n/2+1$.
 Note that $W_s-W_{t_{i-1}} =n(s-t_{i-1})(W_{t_i}-W_{t_{i-1}}) + B_s$ and $\widetilde W_s-\widetilde W_{t_{i-1}} =n(s-t_{i-1})(W_{t_i}-W_{t_{i-1}}) + \widetilde B_s $, and therefore 
 for $\PP^{(X_{t_{i-1}},W_{t_i}-W_{t_{i-1}})}$-almost all $(x,\delta)\in\R^2$,
\begin{equation}\label{i3aaa}
\EE[R_i\,|\,X_{t_{i-1}} = x, W_{t_i}-W_{t_{i-1}}=\delta] =2\int_{t_{i-1}}^{t_i} \int_{t_{i-1}}^{t_i} \text{Cov}(\mu(x+a_s\delta+B_s), \mu(x+a_t\delta+B_t))\,ds\, dt, 
\end{equation}
where $a_s=n(s-t_{i-1})$ for all $s\in [t_{i-1},t_i]$. 
Note further that $B_s $ and $B_t $ are 
nonnegatively
correlated for all $s,t\in [t_{i-1},t_i]$.
 It is well known that  
 nonnegatively correlated, jointly normally distributed 
random
variables are positively associated so that $\text{Cov}(\mu(x+a_s\delta+B_s), \mu(x+a_t\delta+B_t))\ge 0$ for all $s,t\in [t_{i-1},t_i]$ because $\mu$ is increasing.  
 In Lemma~\ref{TongIn} in the appendix we 
establish for bivariate 
normal random
variables $(Z_1,Z_2)$ with 
nonnegative  correlation and increasing, piecewise Lipschitz continuous functions $f_1,f_2\colon \R\to\R$ a lower bound for  the covariance of $f_1(Z_1)$ and $f_2(Z_2)$ in terms of the jump sizes of $f_1$ and $f_2$ at their discontinuity points. 
We then apply these covariance bounds to the integrand in the right hand side of~\eqref{i3aaa} and take expectations to obtain
\begin{equation}\label{i4}
\EE[R_i] \ge  (\mu(\xi_\ell+)-\mu(\xi_\ell-))^2 \,
\frac{c}{n^{5/2}}
\end{equation} 
for all $i\ge n/2+1$ and $\ell\in\{1,\dots,k\}$, where $c\in (0,\infty)$ does not depend on $n$,
see Lemma~\ref{BrBr}  and Lemma~\ref{rest}. 

Combining~\eqref{i0},~\eqref{i1},~\eqref{i2},~\eqref{i3xx} and~\eqref{i4} yields the claimed lower  bound in Theorem~\ref{Thm1} for the $L_2$-error in place of the $L_1$-error.

Our proof technique also applies to obtain lower error bounds for estimating occupation time functionals of the Brownian motion $W$. 
\begin{theorem}\label{Thm2}
Assume that $\mu$ satisfies the assumptions ($\mu$1) and ($\mu$3) and is increasing or decreasing. Then for every $\varepsilon \in (0,\infty)$ there exists $c\in(0, \infty)$ such that for all $n\in\N$, 
\begin{equation}\label{l4}
\inf_{\substack{
       t_1,\dots ,t_n \in [0,1]\\
        g \colon \R^n \to \R \text{ measurable} \\
       }}	 \EE\Bigl[\Bigl|\int_0^1 \mmu(W_s) \, ds-g(W_{t_1}, \ldots, W_{t_n})\Bigr|^p\Bigr]^{1/p}\geq \begin{cases} \frac{c}{n^{3/4}}, & \text{ if }p=2,\\  \frac{c}{n^{3/4+\varepsilon}}, & \text{ if }p=1. \end{cases}
\end{equation}
\end{theorem}

Theorem~\ref{Thm2} generalizes a result in~\cite{NO11}, which establishes the lower bound $c/n^{3/4}$ for the $L_2$-error of the Riemann-sum estimator $n^{-1}\sum_{i=1}^n 1_{[0,\infty)}(W_{(i-1)/n})$ of the integral $\int_0^1 1_{[0,\infty)}(W_s)\, ds$. For that particular estimator,  $c/n^{3/4}$ is also an upper $L_2$-error bound, see~\cite{NO11}.
We conjecture that the latter result extends to our more general setting as well, in the sense that the $L_p$-error of the Riemann-sum estimator $n^{-1}\sum_{i=1}^n \mu(W_{(i-1)/n})$ of  $I_\mu(W) = \int_0^1 \mmu(W_s) \, ds$ is bounded by $c/n^{3/4}$ if $\mu$ satisfies ($\mu$1) and ($\mu$2). We furthermore add that lower 
and upper
error bounds for estimators of occupation time functionals $I_\mu(W)$
for bounded $\mu$ from fractional $L_2$-Sobolev spaces
  are  established in~\cite{A2017}. 
In particular, in 
the latter paper it is shown that for $\mu\in L_2(\R)$ having the Fourier-transform $u\mapsto (1+|u| )^{-s-1/2}$, where $s\in [0,1]$, the lower bound $c/n^{(1+s)/2+}$ holds for the $L_2$-error of any estimator of $I_\mu(W)$ based  on $W_{i/n}, i=1,\dots,n$.  

Obviously, strong approximation at time 1 of the solution $X$ of the equation \eqref{sde0} with $\sigma=1$ is closely related to estimating the occupation time functional $I_\mu(X)=\int_0^1 \mu(X_s)\, ds$. The first problem requires approximation of $I_\mu(X)$ based on $n$ evaluations of the driving Brownian motion $W$, while the second problem deals with   approximation of $I_\mu(X)$ based on $n$ evaluations of the process $X$. It is an open question to us whether these problems have the same complexity in the sense of identical smallest possible error rates in terms of $n$. 

Clearly, Theorem~\ref{Thm2} is not a special case of Theorem~\ref{Thm1} but it seems likely that both results are particular cases of a 
(yet to be shown) result
on sharp lower error bounds for strong approximation of systems of SDEs with commutative noise and drift coefficients that satisfy suitable multivariate versions of the conditions ($\mu$1) to ($\mu$5). The condition of commutative noise stems from the fact that for strong approximation at a single time of systems of SDEs with non-commutative noise the $L_2$-error rate $1/2$ can in general not be improved by any approximation based on finitely many evaluations of the driving Brownian motion. See~\cite{ClarkCameron1980,MG02_habil} for details.   
 
\section{Proofs} \label{Auxil}

We briefly outline the structure of this section. In Subsection~\ref{sub21} we provide properties of the drift coefficient $\mu$ under the assumption ($\mu1$) and the lower bound-techniques that are crucial for the proof of both Theorem~\ref{Thm1} and Theorem~\ref{Thm2}. The latter theorem is then proven in Subsection~\ref{sub22}. In Subsection~\ref{sub23} we collect properties of 
solutions 
of~\eqref{sde0} that are used in the proof of Theorem~\ref{Thm1}. The proof of Theorem~\ref{Thm1} is then carried out in Subsection~\ref{sub24}.

\subsection{Basic properties of $\mu$ and two crucial lower bound-techniques}\label{sub21}
We first provide basic properties of functions satisfying the assumption ($\mu1$). These properties are well-known. For the convenience of the reader we also provide proofs of them.

\begin{lemma}\label{basics}
Assume that $\mu$ satisfies ($\mu1$) and put
\[
D_i = \{(u,v)\in \R^2\colon (u-\xi_i)\,(v-\xi_i)\le 0\}
\]
for $i \in\{ 1,\dots,k\}$.
 Then $\mu$ satisfies a linear growth condition, the limits $\mu(\xi_i-) = \lim_{x\uparrow \xi_i}\mu(x)$ and $\mu(\xi_i+) = \lim_{x\downarrow \xi_i}\mu(x)$
exist for all $i\in\{1,\dots,k\}$, and there exists $c\in (0,\infty)$ such that for all $x,y\in \R$,
\begin{equation}\label{basics1}
|\mu(x) -\mu(y)| \le c\, \Bigl(|x-y| + \sum_{i=1}^k 1_{D_i}(x,y)\Bigr).
\end{equation}
\end{lemma}

\begin{proof}
For every $i\in\{1,\dots,k+1\}$ fix some $x_i\in (\xi_{i-1},\xi_i)$. By ($\mu1$) there exists $c\in (0,\infty)$ such that for all $i\in\{1,\dots,k+1\}$ and all $x\in (\xi_{i-1},\xi_{i})$,
\[
|\mu(x)| \le |\mu(x)-\mu(x_i)| + |\mu(x_i)| \le c\,|x-x_i| +|\mu(x_i)| \le c\,|x| + c\,|x_i| + |\mu(x_i)|.
\]
Hence for all $x\in\R$,
\[
|\mu(x)| \le \max_{i =1,\dots, k}|\mu(\xi_i)| 
+ \max_{i =1,\dots, k+1} (c\,|x_i| +|\mu(x_i)|) + c\,|x|,
\]
which proves that $\mu$ satisfies a linear growth condition.

Next observe that
\[
\R^2 \setminus \bigcup_{i=1}^{k+1} (\xi_{i-1},\xi_i)^2 = \bigcup_{i=1}^k D_i. 
\]
Using the latter fact, the linear growth property of $\mu$ and  ($\mu1$) we see that there exist $c_1,c_2\in (0,\infty)$ such that for all $x,y\in \R$,
\begin{equation*}
\begin{aligned}
|\mu(x)-\mu(y)| & \le c_1\,|x- y| + c_1\,(1 +|x| +|y|)\, 1_{\bigcup_{i=1}^k D_i} (x,y)\\
& \le  c_1\,|x- y| + c_2\,\sum_{i=1}^k (1 +|x-\xi_i| +|y-\xi_i|)\,1_{D_i}(x,y).
\end{aligned}
\end{equation*}
Now observe that 
for all $i\in\{1,\dots,k\}$ and 
all $x,y\in D_i$ we have $|x-\xi_i| \le |x-y|$, which finishes the proof of~\eqref{basics1}.

Finally, let $i\in\{1,\dots,k\}$ and let $(x_n)_{n\in\N}$ be a sequence in $(\xi_{i-1},\xi_i)$, which converges to $\xi_i$. Using the Lipschitz continuity of $\mu$ on $(\xi_{i-1},\xi_i)$ we obtain that $(\mu(x_n))_{n\in\N}$ is a Cauchy-sequence and hence has a limit $z\in\R$. If $(\tilde x_n)_{n\in\N}$ is a further sequence in $(\xi_{i-1},\xi_i)$, which converges to $\xi_i$, then $\lim_{n\to\infty} (x_n-\tilde x_n) = 0$ and by the Lipschitz continuity of $\mu$ on $(\xi_{i-1},\xi_i)$ we conclude that $\lim_{n\to\infty} (\mu(x_n)-\mu(\tilde x_n)) = 0$. Thus the sequence $(\mu(\tilde x_n))_{n\in\N}$ converges to $z$ as well. This proves the existence of the limit $\mu(\xi_i-)\in\R$. The existence of the limit $\mu(\xi_{i}+)$ in $\R$ is shown in the same manner. This completes the proof of the lemma. 
\end{proof}

The following two lemmas are crucial for the proof of both Theorem~\ref{Thm1} and  Theorem \ref{Thm2}. 

Lemma \ref{symm} is an elementary 
consequence of the triangle inequality,
see also~\cite[Lemma 3]{MGRY2018}. 
\begin{lemma}\label{symm}
Let $(\Omega,\F,\PP)$ be a probability space, let
$
  (\Omega_1, \mathcal{A}_1)
$
and
$
  (\Omega_2, \mathcal{A}_2)
$
be measurable spaces
and let
$
  V_1 \colon \Omega\to \Omega_1,
$
$
  V_2, V_2'\colon\Omega\to \Omega_2
$
be random variables such that
\[
  \PP^{ (V_1, V_2) } =
  \PP^{ (V_1, V_2') }.
\]
Then for all $p\in[1, \infty)$ and for all measurable mappings
$
  \Phi \colon \Omega_1\times\Omega_2\to \R
$
and
$
  \varphi\colon \Omega_1\to\R,
$

\[
  \EE\big[
    |\Phi(V_1,V_2)- \varphi(V_1)|^p
  \big]^{1/p}
  \ge
  \frac{ 1 }{ 2 }
  \,
  \EE\big[
    | \Phi( V_1, V_2) - \Phi( V_1, V_2') |^p
  \big]^{1/p}
  .
\]
\end{lemma}

Put
\begin{equation}\label{kappa}
\kappa = \frac{1}{16\pi}\,e^{-6}\, \int_0^{1/\sqrt{3}} \frac{1}{\sqrt{1-x^2}}\, e^{-\frac{24}{1-x^2}} dx.
\end{equation}

\begin{lemma}\label{BrBr}
Let $(\Omega,\F,\PP)$ be a probability space, let $t\in(0,1]$, $B, B'\colon [0,t]\times \Omega \to \R$ be Brownian bridges on $[0,t]$, let $U, V\colon\Omega\to\R$ be random variables and assume that $B, B', U, V$ are independent. Furthermore, let $k\in\N$, let $-\infty=\xi_0<\xi_1<\ldots<\xi_k<\xi_{k+1}=\infty$  and  let $h\colon\R\to\R$ satisfy
\begin{itemize}
\item[(i)] $h$ is increasing or decreasing,
\item[(ii)] $h$ is Lipschitz continuous on the interval $(\xi_{i-1}, \xi_i)$  for all $i\in\{1, \ldots, k+1\}$.
\end{itemize}
Then for every $i\in\{1, \ldots, k\}$ it holds 
\begin{align*}
&\EE\Bigl[\Bigl|\int_0^{t}\bigl(h(U+sV+B_s)-h(U+ sV+B'_s)\bigr) \, ds\Bigr|^2\Bigr]\\
&\qquad\qquad \geq \kappa\, (h(\xi_i+)-h(\xi_i-))^2\,   t^2\, \PP(U\in [\xi_i, \xi_i+\sqrt t])\, \PP(V\in [0, 1/\sqrt t]).
\end{align*}
\end{lemma}

\begin{proof}
Note that all of the limits $h(\xi_i+),h(\xi_i-)$, $i=1,\dots,k$, exist due to the assumption (ii), see Lemma~\ref{basics}.
Put
\[
R=(R_s=U+sV+B_s)_{s\in[0,t]}, \quad R'=(R'_s=U+sV+B'_s)_{s\in[0,t]}
\]
and let
\[
D=\EE\Bigl[\Bigl|\int_0^{t}\bigl(h(R_s)-h(R'_s)\bigr) \, ds\Bigr|^2\Bigr].
\] 
Clearly, $\PP^{(R,R')} = \PP^{(R',R)}$, and therefore
\[
D=2\int_0^t\int_0^t\EE[h(R_r)\, h(R_s)-h(R_r)\, h(R'_s)] \, ds \,dr.
\]

Let $\varphi\colon [0,t]^2\times\R^2\to\R$  be given by
\[
\varphi(s,r,u,v)=\EE[h(B_r+u+rv)\, h(B_s+u+sv)]-\EE[h(B_r+u+rv)]\, \EE[h(B_s+u+sv)].
\]
The independence  of $B, B', U, V$ implies that for all $s,r\in[0, t]$ and $\PP^{(U, V)}
$-almost all $(u, v)\in\R^2$,
\begin{align*}
&\EE[h(R_r)\cdot h(R_s)-h(R_r)\, h(R'_s)\,|\,(U,V)=(u,v)]\\
&\qquad = \EE[h(B_r+u+rv)\, h(B_s+u+sv)-h(B_r+u+rv)\, h(B'_s+u+sv)]\\
&\qquad= \EE[h(B_r+u+rv)\, h(B_s+u+sv)]-\EE[h(B_r+u+rv)]\, \EE[h(B_s+u+sv)]\\
&\qquad=\varphi(s,r,u,v).
\end{align*}
Thus,
\[
D=2 \int_{\R^2}\int_0^t\int_0^t\varphi(s,r,u,v) \, ds \,dr\, \PP^{(U,V)}(d(u,v)).
\]

Below we show that for all $s,r\in (0,t)$  and all $u,v\in\R$,
\begin{equation}\label{o8}
\varphi(s,r,u,v)\geq 0.
\end{equation}
Moreover, we show that for all $i\in\{1, \ldots, k\}$, $s,r\in [t/4, t/2]$, $u\in [\xi_i, \xi_i+\sqrt t]$ and $v\in [0, 1/\sqrt t]$,
\begin{equation}\label{o9}
\varphi(s,r,u,v)\geq 8\kappa \, (h(\xi_i+)-h(\xi_i-))^2.
\end{equation}
Using~\eqref{o8}, ~\eqref{o9} and  the independence of $U$ and $V$ 
we conclude that for all $i\in\{1, \ldots, k\}$,
\begin{align*}
D&\geq 2  \int_0^{1/\sqrt{t}}\int_{\xi_i}^{\xi_i+\sqrt{t}} \int_{t/4}^{t/2} \int_{t/4}^{t/2}8\kappa \, (h(\xi_i+)-h(\xi_i-))^2 \, ds \,dr\,\PP^{U}(du)\,\PP^{V}(dv)\\
&=\kappa\, (h(\xi_i+)-h(\xi_i-))^2  \, t^2\, \PP(U\in [\xi_i, \xi_i+\sqrt t])\, \PP(V\in [0, 1/\sqrt t]),
\end{align*}
which is the statement of the lemma.

It remains to prove \eqref{o8} and \eqref{o9}. 
To this end,  let $s,r\in (0,t)$ and $u,v\in\R$. 
Put 
\[
Z=\tfrac{\sqrt t}{\sqrt{r(t-r)}}B_r, \, Y= \tfrac{\sqrt t}{\sqrt{s(t-s)}}B_s
\]
and define $f,g\colon\R\to \R$ by
\[
f(x) = h\Bigl(\tfrac{\sqrt{r(t-r)}}{\sqrt{t}} x + u + rv\Bigr), \, g(x) = h\Bigl(\tfrac{\sqrt{s(t-s)}}{\sqrt{t}} x + u + sv\Bigr),\quad x\in\R.
\]
It is straightforward to see that $Z,Y,f,g$ satisfy the assumptions in Lemma \ref{TongIn} in the appendix with 
\[
\rho = \EE\Bigl[\tfrac{\sqrt t}{\sqrt{s(t-s)}}B_s \tfrac{\sqrt t}{\sqrt{r(t-r)}}B_r\Bigr]=\tfrac{(t-\max(s,r))\, \min(s,r)}{\sqrt{s(t-s)}\sqrt{r(t-r)}}\in[0,1],
\]
$k=l$ and 
\[
a_i=(\xi_i-u-rv)\,\tfrac{\sqrt t}{\sqrt{r(t-r)}}, \quad b_i=(\xi_i-u-sv)\,\tfrac{\sqrt t}{\sqrt{s(t-s)}}
\]
for $i\in\{1, \ldots, k\}$. Hence by Lemma \ref{TongIn},
 \begin{equation}\label{ss2}
\begin{aligned}
\varphi(s,r,u,v)&\geq\sum_{i=1}^{k} \sum_{j=1}^{k} (h(\xi_i+)-h(\xi_i-))\, (h(\xi_j+)-h(\xi_j-)) \\
&\qquad\qquad\qquad\qquad \times  \frac{1}{2\pi}\,e^{-\frac{a_i^2}{2}}\, \int_0^{\rho} \frac{1}{\sqrt{1-x^2}}\, e^{-\frac{(b_j-a_i x)^2}{2 (1-x^2)}} dx.
\end{aligned}
\end{equation}
The latter bound and the assumption that $h$ is increasing or decreasing yield \eqref{o8}.

Next, let $i\in\{1, \ldots, k\}$. Since $h$ is increasing or decreasing we conclude from \eqref{ss2} in particular, that
\begin{equation}\label{ss1}
\varphi(s,r,u,v)\geq (h(\xi_i+)-h(\xi_i-))^2\,  \frac{1}{2\pi}\,e^{-\frac{a_i^2}{2}}\, \int_0^{\rho} \frac{1}{\sqrt{1-x^2}}\, e^{-\frac{(b_i-a_ix)^2}{2 (1-x^2)}} dx.
\end{equation}

Let   $s,r\in [t/4, t/2]$,  $u\in [\xi_i, \xi_i+\sqrt t]$, $v\in [0, 1/\sqrt t]$.
We then have
\[
\max(|a_i|, |b_i|)\leq \bigl(\sqrt t+\tfrac{t}{2}\,  \tfrac {1}{\sqrt t}\bigr)\, \tfrac{\sqrt t}{\sqrt{\tfrac{t}{4} \, \tfrac{3t}{4} }} =2\sqrt 3.
\]
Thus,
\[
e^{-\frac{a_i^2}{2}}\geq e^{-6}.
\]
Moreover, for all $x\in[0, \rho]$,
\[
e^{-\frac{(b_i-a_ix)^2}{2(1-x^2) }}
\geq e^{-\frac{(|a_i|+|b_i|)^2}{2(1-x^2) }}\geq e^{-\frac{24}{1-x^2}},                                
\]
and
\[
\rho = \tfrac{\sqrt{t-\max(s,r)}\, \sqrt{\min(s,r)}}{\sqrt{\max(s,r)}\,\sqrt{t-\min(s,r)}} \geq \tfrac{\sqrt{\tfrac{t}{2}}\, \sqrt{\tfrac{t}{4}}}{\sqrt{\tfrac{t}{2}}\,\sqrt{\tfrac{3t}{4}}}=\tfrac{1}{\sqrt 3}.
\]
Hence we conclude that
\[
\frac{1}{2\pi}\,e^{-\frac{a_i^2}{2}}\, \int_0^{\rho} \frac{1}{\sqrt{1-x^2}}\, e^{-\frac{(b_i-a_ix)^2}{2 (1-x^2)}} dx\geq \frac{1}{2\pi}\,e^{-6}\, \int_0^{1/\sqrt 3} \frac{1}{\sqrt{1-x^2}}\, e^{-\frac{24}{1-x^2}} dx = 8\kappa.
\]
The latter estimate together with  \eqref{ss1} implies \eqref{o9} and completes the proof of the lemma.
\end{proof}

\subsection{Proof of Theorem \ref{Thm2}}\label{sub22}

In the following let $(\Omega,\F,\PP)$ be a probability space and let $W\colon [0,1]\times \Omega\to \R$ be a standard Brownian motion on $[0,1]$. 

For the proof of Theorem~\ref{Thm2} it suffices to show that for every $\varepsilon\in (0,\infty)$  there exists $c\in (0,\infty) $ such that for all $n\in 2\N$ with $n\ge 6$ and all $t_1, \ldots, t_n\in [0,1]$ with 
\begin{equation}\label{disc}
0<t_1<\ldots<t_n=1
\end{equation}
 and 
\begin{equation}\label{b5}
2/n,4/n,\dots, 1\in\{t_1, \ldots, t_n\}
\end{equation}
we have
\begin{equation}\label{finale}
\inf_{        g \colon \R^n \to \R \text{ measurable} \\
       }	 \EE\Bigl[\Bigl|\int_0^1 \mmu(W_s) \, ds-g(W_{t_1}, \ldots, W_{t_n})\Bigr|^p\Bigr]^{1/p}\geq \begin{cases} \frac{c}{n^{3/4}}, & \text{ if }p=2,\\  \frac{c}{n^{3/4+\varepsilon}}, & \text{ if }p=1.\end{cases}
\end{equation}

In the sequel we fix $n\in 2\N$ with $n\ge 6$ and $t_1, \ldots, t_n\in [0,1]$ with~\eqref{disc} and~\eqref{b5}. Moreover, we put $t_0=0$.
Let $\overline W\colon [0,1]\times \Omega\to\R$ denote the piecewise linear interpolation of $W$ on $[0,1]$ at the points $t_0, \ldots, t_{n}$, i.e. 
\[
\overline W_t=\tfrac{t-t_{i-1}}{t_i-t_{i-1}}\,W_{t_i}+\tfrac{t_i-t}{t_i-t_{i-1}}\, W_{t_{i-1}}, \quad t\in [t_{i-1}, t_i],
\]
for $i\in\{1, \ldots, n\}$, and  put
\[
B=W-\overline W.
\]
Observe that $(B_t)_{t\in [t_{i-1}, t_i]}$ is a Brownian bridge on $[t_{i-1}, t_i]$ for every $i\in\{1, \ldots, n\}$. Furthermore, 
$(B_t)_{t\in [t_{0}, t_1]}, \ldots,(B_t)_{t\in [t_{n-1}, t_{n}]}, \overline W $  are independent. 
Without loss of generality we may assume that $(\Omega,\mathcal F,\PP)$ is rich enough to carry 
 for every $i\in\{1, \ldots, n\}$ a Brownian bridge $(\widetilde B_t)_{t\in [t_{i-1}, t_i]}$   on $[t_{i-1}, t_i]$
such that 
$(\widetilde B_t)_{t\in [t_{0}, t_1]}, \ldots,(\widetilde B_t)_{t\in [t_{n-1}, t_{n}]}, W $ are independent. Put $\widetilde B=(\widetilde B_t)_{t\in[0,1]}$ and
 define a Brownian motion  $\widetilde W\colon [0,1]\times \Omega\to\R$  by
\[
\widetilde W= \overline W+\widetilde B.
\]

\begin{lemma}\label{neu1}
Assume that $\mu$ satisfies ($\mu1$). Then for all measurable $g\colon\R^{n}\to\R$ and all $p\in[1,\infty)$, 
\begin{equation}\label{L1}
\begin{aligned}
  \EE\Bigl[\Bigl|\int_0^1 \mmu(W_s) \, ds-g(W_{t_1}, \ldots, W_{t_n})\Bigr|^p\Bigr]^{1/p}
   \geq \frac{1}{2} \, 
   \EE\Bigl[\Bigl|\int_0^1 \mmu(W_s) \, ds-\int_0^1 \mmu(\widetilde W_s) \, ds\Bigr|^p\Bigr]^{1/p}.
\end{aligned}
\end{equation}
\end{lemma}
\begin{proof} For convenience of the reader we first show that $\int_0^1 \mu(W_s)\, ds$ is well-defined and a random variable. By ($\mu1$) the function $\mu$ is piecewise continuous and therefore Borel-measurable. Thus, for every $f\in C([0,1],\R)$ the function $\mu\circ f\colon [0,1]\to \R$ is Borel-measurable. Moreover, by Lemma~\ref{basics} there exists $c\in (0,\infty)$ such that for every $t\in[0,1]$ we have $|\mu(f(t))|\le c(1+|f(t)|) \le c(1+\|f\|_\infty)$, which shows that $\mu\circ f$ is bounded. Hence $\int_0^1 \mu(f(s))\, ds$ exists for every $f\in C([0,1],\R)$ and therefore, $\int_0^1 \mu(W_s)\, ds\colon \Omega\to \R$ is well-defined. Since $W\colon [0,1]\times \Omega\to \R$ is measurable we have measurability of $\int_0^1 \mu(W_s)\, ds$ as claimed.  

Next we show that there exists a Borel-measurable function $\Phi\colon C([0,1],\R)\times C([0,1],\R)\to\R$  such that 
\begin{equation}\label{v1}
\int_0^1 \mmu(W_s) \, ds=\Phi(\overline W, B)\,\text{ and }\, \int_0^1 \mmu(\widetilde W_s) \, ds=\Phi(\overline W, \widetilde B)\,\text{ almost surely}.
\end{equation}
To this end put $\mathcal D = \{\xi_1,\dots,\xi_k\}$ and note that the function $[0,1]\times C([0,1],\R)\ni (s,f)\mapsto f(s)\in \R$ is continuous. Thus, the mapping $[0,1]\times C([0,1],\R)\ni (s,f)\mapsto 1_\mathcal{D}(f(s))\in \R$ is Borel-measurable, which implies the Borel-measurability of the mapping
\[
T\colon C([0,1],\R)\to \R,\,\, f\mapsto \int_0^1 1_\mathcal{D}(f(s))\, ds.
\]
Hence, $T^{-1}(\{0\})$ is a Borel-subset of $C([0,1],\R)$. Let $f\in  T^{-1}(\{0\})$. Then $\lambda(\{s\in[0,1]\colon f(s)\in \mathcal D\}) = 0$. By ($\mu1$) we know that $\{s\in[0,1]\colon \mu\circ f\text{ is discontinuous in }s\}\subset \{s\in[0,1]\colon f(s)\in \mathcal D\}$. Hence $\mu\circ f$ is a bounded Riemann-integrable function. Thus
\begin{equation}\label{v2}
\forall f\in T^{-1}(\{0\})\colon\,\, \int_0^1 \mu(f(s))\, ds = \lim_{m\to \infty} R_m(f)
\end{equation}
with
\[
R_m\colon C([0,1],\R)\to \R,\,\, f\mapsto \frac{1}{m}\sum_{i=1}^m \mu(f(i/m)).
\]
Clearly, the mappings $R_m, m\in\N$, are Borel-measurable, and therefore the mappings
\[
S_m:= R_m\cdot 1_{T^{-1}(\{0\})}\colon C([0,1],\R)\to \R,\,m\in\N,                            
\] 
are Borel-measurable as well. Using~\eqref{v2} we obtain that the limit
\[
S:=\lim_{m\to\infty}S_m\colon C([0,1],\R)\to \R
\]
exists, is Borel-measurable and satisfies $S(f) = \int_0^1\mu(f(s))\, ds$ for all $f\in T^{-1}(\{0\})$. Note that for $V=W,\widetilde W$,
\[
\EE\Bigl[\int_0^1 1_\mathcal{D} (V_s)\, ds \Bigr] = \int_0^1 \PP(V_s\in \mathcal D)\, ds = 0.
\]
Hence $\PP(V\in  T^{-1}(\{0\})) = 1$ and we conclude that $S(V) = \int_0^1 \mmu(V_s) \, ds$ almost surely. Thus,~\eqref{v1} holds  for the Borel-measurable function 
\[
\Phi\colon C([0,1],\R)\times C([0,1],\R)\to\R,\,\, (f,g) \mapsto S(f+g).
\]

Let $g\colon\R^{n}\to\R$ be measurable. Clearly, there exists a measurable function $\varphi\colon C([0,1],\R)\to\R$ such that
\begin{equation}\label{z3}
g(W_{t_1}, \ldots, W_{t_{n}})=\varphi(\overline W).
\end{equation}
Since $\overline W$ and $B$ are independent, $\overline W$ and $\widetilde B$ are independent, and $\PP^{B}=\PP^{\widetilde B}$, we have
\begin{equation}\label{eqdis}
\PP^{(\overline W, B)}=\PP^{(\overline W, \widetilde B)}.
\end{equation}
We may thus apply Lemma \ref{symm} with $\Omega_1=\Omega_2=C([0,1],\R)$, $V_1=\overline W$, $V_2=B$, $V_2'=\widetilde B$, $\Phi$ as in~\eqref{v1} and $ \varphi$ as in~\eqref{z3} to obtain~\eqref{L1} for every $p\in[1,\infty)$.
 \end{proof}

Throughout the following we use $c,c_1,c_2,\ldots \in (0,\infty)$ to denote positive constants that may change their values in every appearance but neither depend on $n$ nor on the discretization points $t_0,\dots,t_n$.

Next, we provide an upper bound for  the right hand side of~\eqref{L1} in the case $p=4$.

\begin{lemma}\label{lemmax1} Assume that $\mu$ satisfies ($\mu1$). Then for every $\delta\in (0,\infty)$ there exists $c\in (0,\infty)$ such that
\[
 \EE\Bigl[\Bigl|\int_0^1 (\mmu(W_s) - \mmu(\widetilde W_s)) \, ds\Bigr|^4\Bigr]^{1/4} \le 
 \frac{c}{n^{3/4-\delta}}.
\]
\end{lemma}

\begin{proof}

 For all $i\in \{1,\dots,n\}$ put
\begin{equation}\label{nn1}
J_i = \int_{t_{i-1}}^{t_i} (\mmu(W_s) - \mmu(\widetilde W_s)) \, ds.
\end{equation}
First, we show that for all $m\in\N$, $i\in \{1,\dots,n\}$ and $j_1,\dots,j_m \in\{1,\dots,n\}\setminus\{i\}$ we have
\begin{equation}\label{v0}
\EE[J_{i}\cdot J_{j_1}\cdots J_{j_m}]  = 0.
\end{equation}

To this end let
$m\in\N$, $i\in \{1,\dots,n\}$ and $j_1,\dots,j_m \in\{1,\dots,n\}\setminus\{i\}$.
We note that by the construction of $\widetilde W$ and the independence of $B,\widetilde B,\overline W$ we have for $\PP^{\overline W}$-almost all $y\in C([0,1],\R)$,
\begin{equation}\label{vv0}
\PP^{(W,\widetilde W)|\overline W = y} = \PP^{(y+B,y+\widetilde B)}.
\end{equation}
Moreover, the processes 
\begin{equation}\label{vvv0}
(B_s,\widetilde B_s)_{s\in[t_{i-1},t_i]}\,\text{ and } \, (B_s,\widetilde B_s)_{s\in[0,1]\setminus [t_{i-1},t_i]}\,\text{  are independent}.
\end{equation}
 Consequently, for $\PP^{\overline W}$-almost all $y\in C([0,1],\R)$,
\begin{equation*}
\begin{aligned}
& \EE[J_{i}\cdot J_{j_1}\cdots J_{j_m}|\overline W = y]\\
 & \qquad= \EE\Bigl[\int_{t_{i-1}}^{t_i} (\mu(y_s +B_s) - \mu(y_s+\widetilde B_s)) \, ds\Bigr] \, \EE\Bigl[\prod_{\ell=1}^m\int_{t_{j_\ell-1}}^{t_{j_\ell}} (\mu(y_s +B_s) - \mu(y_s+\widetilde B_s)) \, ds\Bigr].
 \end{aligned}
\end{equation*}
Furthermore, since $\PP^B=\PP^{\widetilde B}$ we have
\[
\EE\Bigl[\int_{t_{i-1}}^{t_i} (\mu(y_s +B_s) - \mu(y_s+\widetilde B_s)) \, ds\Bigr] = \int_{t_{i-1}}^{t_i}( \EE[ \mu(y_s +B_s) ]- \EE[\mu(y_s+\widetilde B_s) ])\, ds = 0.
\]
Combining the latter two equalities and taking expectation with respect to $\PP^{\overline W}$ yields~\eqref{v0}. 

Next, put
\[
i^* = \min\{i\in\{1,\dots,n\}\colon t_{i} > 2/n\}
\]
and for all $i\in\{i^*,\dots,n-1\}$ put
\[
\ell_i = \max\{j\in\{i+1,\dots,n\}\colon t_j \le t_{i-1} + 4/n\}.
\]
Note that~\eqref{b5} and the assumption
$n\ge 6$
 imply that $i^*\le n-1$ and $t_{i^*-1} = 2/n$. Moreover,~\eqref{b5} implies that $t_{i+1}-t_{i-1} \le 4/n$ for all $i\in\{1,\dots,n-1\}$, which shows that all numbers $\ell_i$ are  well-defined.

Clearly,
\begin{equation}\label{c11}
\begin{aligned}
 \EE\Bigl[\Bigl|\int_{0}^1 (\mmu(W_s) - \mmu(\widetilde W_s)) \, ds\Bigr|^4\Bigr]& \le  4\,\EE\Bigl[\Bigl|\int_0^{2/n} (\mmu(W_s) - \mmu(\widetilde W_s)) \, ds\Bigr|^4\Bigr] + 4\,\EE\Bigl[\Bigl( \sum_{i=i^*}^n J_i\Bigr)^4\Bigr].
 \end{aligned}
\end{equation}
Moreover, with the help of~\eqref{v0} we obtain
\begin{equation}\label{d11x}
\begin{aligned}
&\EE\Bigl[\Bigl( \sum_{i=i^*}^n J_i\Bigr)^4\Bigr]\\ &
\qquad = \EE\Bigl[\sum_{i,j=i^*}^n J_i^2J_j^2\Bigr] =\sum_{i=i^*}^n \EE[J_i^4] + 2\sum_{i=i^*}^{n-1}\sum_{j=i+1}^n \EE[J_i^2J_j^2]\\
&\qquad =\sum_{i=i^*}^n  \EE[J_i^4] + 2\sum_{i=i^*}^{n-1}\EE\Bigl[J_i^2\Bigl(\sum_{j=i+1}^{\ell_i}J_j\Bigr)^2\Bigr] + 2\sum_{i=i^*}^{n-1}\sum_{j=\ell_i +1}^n \EE[J_i^2J_j^2]\\
& \qquad\le \sum_{i=i^*}^n \EE[J_i^4] + 2\sum_{i=i^*}^{n-1}\EE\Bigl[J_i^2\Bigl(\int_{t_i}^{t_{i-1}+4/n}|\mmu(W_s) - \mmu(\widetilde W_s)| \, ds\Bigr)^2\Bigr] 
+ 2\sum_{i=i^*}^{n-1}\sum_{j=\ell_i+1}^n \EE[J_i^2J_j^2].
 \end{aligned}
\end{equation}
Since $\mu$ satisfies a linear growth condition, see Lemma~\ref{basics}, and $\sup_{s\in[0,1]}\EE[|W_s|^4] <\infty$ we obtain that there exists $c\in (0,\infty)$ such that for all $0\le a \le b\le 1$, 
\[
 \EE\Bigl[\Bigl(\int_a^b|\mmu(W_s) - \mmu(\widetilde W_s)| \, ds\Bigr)^4\Bigr] \le c\,(b-a)^4.
\]
Employing the latter fact and H\"older's inequality and observing~\eqref{b5} we conclude from~\eqref{c11} and~\eqref{d11x} that there exist $c_1,c_2\in (0,\infty)$ such that
\begin{equation}\label{ux1}
\begin{aligned}
& \EE\Bigl[\Bigl|\int_0^1 (\mmu(W_s) - \mmu(\widetilde W_s)) \, ds\Bigr|^4\Bigr]\\  &\qquad \le c_1\,\Bigl((2/n)^4 
+\sum_{i=i^*}^n(t_i-t_{i-1})^4
+\sum_{i=i^*}^{n-1}(t_i-t_{i-1})^2 \cdot 
(t_{i-1}+4/n-t_i)^2\Bigr)\\
  &\qquad\qquad+8\sum_{i=i^*}^{n-1}\sum_{j=\ell_i+1}^n \EE[J_i^2J_j^2] \\
& \qquad \le \frac{c_2}{n^3} + 8\sum_{i=i^*}^{n-1}\sum_{j=\ell_i+1}^n \EE[J_i^2J_j^2].  
 \end{aligned}
\end{equation}

Note that for all $i\in\{i^*,\dots,n-1\}$ and $j\in\{\ell_i+1,\dots,n\}$ we have
\begin{equation}\label{extra}
t_{j-1} \ge t_j-2/n >t_{i-1}+2/n.
\end{equation}
Below we show that there exists $c\in (0,\infty)$ such that for  all $i\in\{i^*,\dots,n-1\}$ and $j\in\{\ell_i+1,\dots,n\}$,
\begin{equation}\label{nn3}
\begin{aligned}
\EE[J_i^2\, J_j^2]& \le c \,\Bigl(\frac{1}{n^5} +    \frac{\ln(n+1)}{n^3}\cdot\frac{(t_{i}-t_{i-1})(t_j-t_{j-1})}{\sqrt{(t_{i}-2/n)(t_{j}-2/n-t_{i-1})}}\Bigr).
\end{aligned}
\end{equation} 
Employing~\eqref{ux1} to~\eqref{nn3} we conclude that there exist $c_1,c_2\in (0,\infty)$ such that 
\begin{equation*}
\begin{aligned}
& \EE\Bigl[\Bigl|\int_0^1 (\mmu(W_s) - \mmu(\widetilde W_s)) \, ds\Bigr|^4\Bigr] \\
 &\qquad\qquad \le c_1\,\Bigl(\frac{1}{n^3} +  \frac{\ln(n+1)}{n^3}\,\sum_{i=i^*}^{n-1}\sum_{j=\ell_i+1}^n\frac{(t_{i}-t_{i-1})(t_j-t_{j-1})}{\sqrt{(t_{i}-2/n)(t_{j}-2/n-t_{i-1})}} \Bigr)\\
 & \qquad\qquad \le c_1\,\Bigl(\frac{1}{n^3} +   \frac{\ln(n+1)}{n^3}\,\sum_{i=i^*}^{n-1}\frac{t_{i}-t_{i-1}}{\sqrt{t_{i}-2/n}}\int_{t_{i-1}+2/n}^1 \frac{1}{\sqrt{y-2/n-t_{i-1}}}\, dy\Bigr)\\
 & \qquad\qquad \le 2c_1\,\Bigl(\frac{1}{n^3} +  \frac{\ln(n+1)}{n^3}\,\int_{2/n}^1 \frac{1}{\sqrt{x-2/n}}\, dx\Bigr) \le c_2 \, \frac{\ln(n+1)}{n^3}.
\end{aligned}
\end{equation*}
The latter estimate clearly implies the statement of the lemma.

It remains to prove~\eqref{nn3}.
 We first show that for all $i,j\in\{1,\dots,n\}$ with $i\neq j$ we have
\begin{equation}\label{p1}
\PP^{(W_s)_{s\in[t_{i-1},t_i]},(\widetilde W_s)_{s\in[t_{j-1},t_j]}} = \PP^{(W_s)_{s\in[t_{i-1},t_i]},( W_s)_{s\in[t_{j-1},t_j]}}.
\end{equation}
To this end we use~\eqref{vv0} and~\eqref{vvv0} to obtain
 that for
$\PP^{\overline W}$-almost all $y\in C([0,1],\R)$,
\begin{equation*}
\begin{aligned}
\PP^{(W_s)_{s\in[t_{i-1},t_i]},(\widetilde W_s)_{s\in[t_{j-1},t_j]}|\overline W = y} & = \PP^{(y_s+B_s)_{s\in[t_{i-1},t_i]},(y_s+\widetilde B_s)_{s\in[t_{j-1},t_j]}}\\
& = \PP^{(y_s+B_s)_{s\in[t_{i-1},t_i]}}\times \PP^{(y_s+\widetilde B_s)_{s\in[t_{j-1},t_j]}}\\
&= \PP^{(y_s+B_s)_{s\in[t_{i-1},t_i]}}\times \PP^{(y_s+ B_s)_{s\in[t_{j-1},t_j]}}\\
&= \PP^{(y_s+B_s)_{s\in[t_{i-1},t_i]},(y_s+ B_s)_{s\in[t_{j-1},t_j]}}\\
 & = \PP^{(W_s)_{s\in[t_{i-1},t_i]},( W_s)_{s\in[t_{j-1},t_j]}|\overline W = y},
\end{aligned}
\end{equation*}
which clearly implies~\eqref{p1}. Next, recall that $W$ and $\widetilde W$ coincide at the points $t_0,\dots,t_n$. The latter fact and~\eqref{p1} imply that for all $i,j\in\{1,\dots,n\}$ with $i\neq j$,
\begin{equation}\label{u1}
\begin{aligned}
\EE[J_i^2 J_j^2 ]
&  \le 4\, \EE\Bigl[\Bigl(\Bigl(\int_{t_{i-1}}^{t_i} (\mu(W_s)-\mu(W_{t_{i-1}}))\, ds\Bigr)^2 +\Bigl(\int_{t_{i-1}}^{t_i} (\mu(\widetilde W_s)-\mu(\widetilde W_{t_{i-1}}))\, ds\Bigr)^2\Bigr)\\
& \quad\qquad \times \Bigl(\Bigl(\int_{t_{j-1}}^{t_j} (\mu(W_s)-\mu(W_{t_{j-1}}))\, ds\Bigr)^2 +\Bigl(\int_{t_{j-1}}^{t_j} (\mu(\widetilde W_s)-\mu(\widetilde W_{t_{j-1}}))\, ds\Bigr)^2\Bigr)\Bigr]\\
& = 16\, \EE\Bigl[\Bigl(\int_{t_{i-1}}^{t_i} (\mu(W_s)-\mu(W_{t_{i-1}}))\, ds\Bigr)^2\, \Bigl(\int_{t_{j-1}}^{t_j} (\mu(W_s)-\mu(W_{t_{j-1}}))\, ds\Bigr)^2\Bigr].
\end{aligned}
\end{equation}

By~\eqref{basics1} in Lemma~\ref{basics} and~\eqref{b5} we see that there exists $c\in (0,\infty)$ such that for all $i\in\{1,\dots,n\}$,
\begin{equation}\label{u3}
\begin{aligned}
       \Bigl|\int_{t_{i-1}}^{t_i} (\mu(W_s)-\mu(W_{t_{i-1}}))\, ds \Bigr|&
\le c\,\Bigl(\int_{t_{i-1}}^{t_i}|W_s-W_{t_{i-1}}|\, ds +\int_{t_{i-1}}^{t_i} \sum_{\ell=1}^k 1_{D_\ell}(W_s,W_{t_{i-1}})\, ds \Bigr)
\end{aligned}
\end{equation}
as well as
\begin{equation}\label{uv3}
\int_{t_{i-1}}^{t_i} \sum_{\ell=1}^k 1_{D_\ell}(W_s,W_{t_{i-1}})\, ds\le \frac{c}{n}
\end{equation}
and
\begin{equation}\label{uv30}
\EE\Bigl[\Bigl(\int_{t_{i-1}}^{t_i} |W_s-W_{t_{i-1}}|\, ds\Bigr)^2\Bigr]
\le (t_i-t_{i-1})^3 \le \frac{c}{n^3}.
\end{equation}
Using~\eqref{u3} and~\eqref{uv3} we conclude that  there exists $c\in (0,\infty)$ such that for all $i,j\in\{1,\dots,n\}$,
\begin{equation}\label{uv1}
\begin{aligned}
& \Bigl(\int_{t_{i-1}}^{t_i} (\mu(W_s)-\mu(W_{t_{i-1}}))\, ds\Bigr)^2\, \Bigl(\int_{t_{j-1}}^{t_j} (\mu(W_s)-\mu(W_{t_{j-1}}))\, ds\Bigr)^2 \\
& \qquad \le c\,  \Bigl(\int_{t_{i-1}}^{t_i} |W_s-W_{t_{i-1}}|\, ds\Bigr)^2\, \Bigl(\int_{t_{j-1}}^{t_j} |W_s-W_{t_{j-1}}|\, ds\Bigr)^2 \\
& \qquad\qquad + \frac{c}{n^2}\,\Bigl( \Bigl(\int_{t_{i-1}}^{t_i} |W_s-W_{t_{i-1}}|\, ds\Bigr)^2 +
\Bigl(\int_{t_{j-1}}^{t_j} |W_s-W_{t_{j-1}}|\, ds\Bigr)^2  \Bigr) \\
& \qquad\qquad + \frac{c}{n^2} \sum_{\ell,r=1}^k \int_{t_{i-1}}^{t_i}  1_{D_r}(W_s,W_{t_{i-1}})\, ds
 \int_{t_{j-1}}^{t_j} 1_{D_\ell}(W_s,W_{t_{j-1}})\, ds.
\end{aligned}
\end{equation}
Employing~\eqref{uv30} and~\eqref{uv1} 
and observing the fact that for all $\ell\in\{1, \ldots, k\}$ and all $(u,v)\in D_{\ell}$ 
\[
|u-v|=|(u-\xi_{\ell})-(v-\xi_{\ell})|\geq |v-\xi_{\ell}|
\]
we obtain that  there exist $c_1,c_2\in (0,\infty)$ such that for all $i,j\in\{1,\dots,n\}$ with $i\neq j$, 
\begin{equation}\label{uv2}
\begin{aligned}
& \EE\Bigl[\Bigl(\int_{t_{i-1}}^{t_i} (\mu(W_s)-\mu(W_{t_{i-1}}))\, ds\Bigr)^2\, \Bigl(\int_{t_{j-1}}^{t_j} (\mu(W_s)-\mu(W_{t_{j-1}}))\, ds\Bigr)^2 \Bigr]\\
& \qquad \le c_1\,  \EE\Bigl[\Bigl(\int_{t_{i-1}}^{t_i} |W_s-W_{t_{i-1}}|\, ds\Bigr)^2\Bigr]\, \EE\Bigl[\Bigl(\int_{t_{j-1}}^{t_j} |W_s-W_{t_{j-1}}|\, ds\Bigr)^2 \Bigr]\\
& \qquad\qquad + \frac{c_1}{n^2}\, \EE\Bigl[ \Bigl(\int_{t_{i-1}}^{t_i} |W_s-W_{t_{i-1}}|\, ds\Bigr)^2 +
\Bigl(\int_{t_{j-1}}^{t_j} |W_s-W_{t_{j-1}}|\, ds\Bigr)^2  \Bigr] \\
& \qquad\qquad + \frac{c_1}{n^2} \sum_{\ell,r=1}^k \int_{t_{i-1}}^{t_i} \int_{t_{j-1}}^{t_j} \PP\bigl(|W_s-W_{t_{i-1}}| \ge |W_{t_{i-1}}-\xi_\ell|,\\
&  \qquad\qquad  \qquad\qquad   \qquad\qquad \qquad\qquad |W_t-W_{t_{j-1}}| \ge |W_{t_{j-1}}-\xi_r|\bigr)  \, dt\, ds\\
& \qquad \le \frac{c_2}{n^5} + \frac{c_1}{n^2} \sum_{\ell,r=1}^k \int_{t_{i-1}}^{t_i} \int_{t_{j-1}}^{t_j} \PP\bigl(|W_s-W_{t_{i-1}}| \ge |W_{t_{i-1}}-\xi_\ell|,\\
&  \qquad\qquad  \qquad\qquad   \qquad\qquad \qquad\qquad |W_t-W_{t_{j-1}}| \ge |W_{t_{j-1}}-\xi_r|\bigr)  \, dt\, ds.
\end{aligned}
\end{equation}

Put
\[
\alpha_n = 2 \sqrt{\ln(n+1)/n}.
\]
Observing~\eqref{b5} we obtain by standard estimates for Gaussian probabilities that
for all $\ell,r\in\{1,\dots,k\}$, all $i,j\in\{1,\dots,n\}$ with $i< j$ and all $s\in[t_{i-1},t_i]$, $t\in[t_{j-1},t_j]$ we have
\begin{equation}\label{more1}
\begin{aligned}
& \PP\bigl(|W_s-W_{t_{i-1}}| \ge |W_{t_{i-1}}-\xi_\ell|, |W_t-W_{t_{j-1}}| \ge |W_{t_{j-1}}-\xi_r|\bigr)\\
& \qquad\quad \le  \PP\bigl(|W_s-W_{t_{i-1}}| \ge \alpha_n\bigr) +  \PP\bigl(|W_t-W_{t_{j-1}}| \ge \alpha_n\bigr)\\
& \qquad\qquad \qquad + \PP\bigl( |W_{t_{i-1}}-\xi_\ell|\le \alpha_n, |W_{t_{j-1}}-\xi_r|\le \alpha_n\bigr)\\
& \qquad\quad \le 4 \PP\bigl(W_1 \ge \sqrt{n/2}\,\alpha_n\bigr) + \frac{2\alpha_n^2}{\pi\,\sqrt{t_{i-1}\,(t_{j-1}-t_{i-1})}}\\
& \qquad\quad \le 4  \PP\bigl(W_1 \ge \sqrt{2\,\ln(n+1)}\bigr) + \frac{8\ln(n+1)}{\pi\,n\sqrt{t_{i-1}\,(t_{j-1}-t_{i-1})}}\\
& \qquad\quad \le \frac{4}{\sqrt{4\pi\,\ln(n+1)}(n+1)}+   \frac{8\ln(n+1)}{\pi\,n\sqrt{t_{i-1}\,(t_{j-1}-t_{i-1})}}.
\end{aligned}
\end{equation}
Using~~\eqref{b5} and~\eqref{extra} 
we conclude from~\eqref{more1}  that there exists $c\in (0,\infty)$ such that for all $\ell,r\in\{1,\dots,k\}$, all $i\in\{i^*,\dots,n-1\}$ and $j\in\{\ell_i+1,\dots,n\}$  we have
\begin{equation}\label{more2}
\begin{aligned}
& \int_{t_{i-1}}^{t_i} \int_{t_{j-1}}^{t_j} \PP\bigl(|W_s-W_{t_{i-1}}| \ge |W_{t_{i-1}}-\xi_\ell|,|W_t-W_{t_{j-1}}| \ge |W_{t_{j-1}}-\xi_r|\bigr)  \, dt\, ds\\
& \qquad\quad \le \frac{c}{n^3} + \frac{c\,\ln(n+1)}{n}\,\frac{(t_i-t_{i-1})(t_j-t_ {j-1})}{\sqrt{(t_{i}-2/n)\,(t_{j}-2/n-t_{i-1})}}.
\end{aligned}
\end{equation}
Finally, combining~\eqref{u1},~\eqref{uv2} and~\eqref{more2} yields~\eqref{nn3}, which completes the proof of the lemma.
\end{proof}

We 
proceed by providing a lower bound for the right hand side of
 \eqref{L1} for the case $p=2$.

 \begin{lemma}\label{lemmax2} Assume that $\mu$ satisfies ($\mu1$), ($\mu3$) and is increasing or decreasing. 
 Then there exists $c\in(0, \infty)$ such that  
\begin{equation}\label{lb11}
 \EE\Bigl[\Bigl|\int_0^1 (\mmu(W_s) - \mmu(\widetilde W_s)) \, ds\Bigr|^2\Bigr] \ge 
 \frac{c}{n^{3/2}}.
\end{equation}
 \end{lemma}
 
 \begin{proof}
 Recall the definition~\eqref{nn1} of $J_1,\dots,J_n$ in the proof of Lemma~\ref{lemmax1}. By~\eqref{v0} we have
 \begin{equation}\label{bb1}
  \EE\Bigl[\Bigl|\int_0^1 (\mmu(W_s) - \mmu(\widetilde W_s)) \, ds\Bigr|^2\Bigr] = \sum_{i=1}^n\EE[J_i^2].
 \end{equation}

Fix $i\in\{1,\ldots,n\}$,  observe that for all $s\in[t_{i-1}, t_i]$ it holds
\[
\overline W_s = W_{t_{i-1}}+\tfrac{s-t_{i-1}}{t_i-t_{i-1}}\,(W_{t_i}-W_{t_{i-1}})
\]
and put
\[
U=W_{t_{i-1}}, \quad V=\tfrac{1}{t_i-t_{i-1}}\,(W_{t_i}-W_{t_{i-1}}).
\]
We then have
\[
J_i=\int_{0}^{t_i-t_{i-1}}(\mmu(U+sV+B_{t_{i-1}+s})-\mmu(U+sV+\widetilde B_{t_{i-1}+s})) \, ds.
\]
Choose $\ell\in\{1,\dots,k\}$ according to condition ($\mu3$),
i.e. $\mu(\xi_{\ell}+)\not=\mu(\xi_{\ell}-)$.
Applying Lemma \ref{BrBr} we may then conclude  that
\begin{equation}\label{z5}
\begin{aligned}
 \EE[J_i^2]\geq c_1 \, (t_i-t_{i-1})^2\, \PP(W_{t_{i-1}}\in [\xi_\ell, \xi_\ell+\sqrt {t_i-t_{i-1}}])\, \PP(W_{t_i}-W_{t_{i-1}}\in [0, \sqrt {t_i-t_{i-1}}]),
\end{aligned}
\end{equation}
where 
\[
c_1=\kappa\, (\mmu(\xi_\ell+)-\mmu(\xi_\ell-))^2 >0
\] 
and $\kappa$ is given by \eqref{kappa}.
   Moreover,
\begin{equation}\label{h1}
\PP(W_{t_i}-W_{t_{i-1}}\in [0, \sqrt {t_i-t_{i-1}}])=\frac{1}{\sqrt{2\pi}}\int_0^1 e^{-\frac{x^2}{2}}dx\geq \frac{1}{\sqrt{2\pi e}}.
\end{equation}
Furthermore, if $t_{i-1}\ge 1/2$ then
\begin{equation}\label{z7}
\begin{aligned}
\PP(W_{t_{i-1}}\in [\xi_\ell, \xi_\ell+\sqrt {t_i-t_{i-1}}]) & = \int_{\xi_\ell/\sqrt{t_{i-1}}}^{(\xi_\ell+\sqrt{t_i-t_{i-1}})/\sqrt{t_{i-1}}}\frac{1}{\sqrt{2\pi}}e^{-x^2/2}\,dx\\
&\ge \frac{1}{\sqrt{2\pi}} e^{-(|\xi_\ell|+1)^2}\sqrt {t_i-t_{i-1}}.
\end{aligned}
\end{equation}

Let $r\in\{1, \ldots, n\}$ satisfy $t_r=1/2$. 
Using~\eqref{z5} to~\eqref{z7} we conclude that there exists $c\in (0,\infty)$ such that
\begin{equation}\label{z6}
\begin{aligned}
\sum_{i=1}^n \EE[J_i^2]\ge c\, \sum_{i=r+1}^n (t_i-t_{i-1})^{5/2}.
\end{aligned}
\end{equation}

By  H\"older's inequality,
\[
\frac {1}{2}=\sum_{i=r+1}^n  (t_i-t_{i-1})
\leq 
n^{3/5}\cdot \Bigl(\sum_{i=r+1}^n  (t_i-t_{i-1})^{5/2}\Bigr)^{2/5}.
\]
Thus,
\begin{equation}\label{oo33}
\sum_{i=r+1}^n  (t_i-t_{i-1})^{5/2}\geq \frac{1}{2^{5/2}n^{3/2}}.
\end{equation}
Hence there exists $c\in (0,\infty)$ such that
\[
\sum_{i=1}^n \EE[J_i^2]\geq c\,\frac{1}{n^{3/2}}.
\]
Combining~\eqref{bb1} and the latter inequality completes the proof of the lemma.
\end{proof}

We are ready to establish the estimate~\eqref{finale}. Clearly, the lower bound in~\eqref{finale} for the case $p=2$ is a consequence of~\eqref{L1} in Lemma~\ref{neu1} with $p=2$ and Lemma~\ref{lemmax2}.

For the case $p=1$ put
\[
Z = \int_0^1 (\mmu(W_s) - \mmu(\widetilde W_s)) \, ds
\]
and let $\delta\in (0,\infty)$.
Using Lemma~\ref{lemmax1} and Lemma~\ref{lemmax2} we obtain by H\"older's inequality that there exist $c_1,c_2\in (0,\infty)$ such that
\begin{equation}\label{star1}
c_1 n^{-3/2}\le \EE[Z^2] = \EE[|Z|^{2/3}\cdot|Z|^{4/3} ]\le \EE[|Z|]^{2/3}\cdot \EE[|Z|^4]^{1/3}
\le \EE[|Z|]^{2/3}\cdot \bigl(c_2 n^{-3+4\delta}\bigr)^{1/3}.
\end{equation}
Hence
\begin{equation}\label{star2}
\EE[|Z|] \ge  c_1^{3/2} n^{-9/4}\cdot c_2^{-1/2} n^{3/2-2\delta}= 
c_1^{3/2}c_2^{-1/2} n^{-3/4-2\delta}.
\end{equation}
The latter estimate with $\delta=\varepsilon/2$ and~\eqref{L1} in Lemma~\ref{neu1} with $p=1$ yield the lower bound in~\eqref{finale} for the case $p=1$,
which completes
the proof of~\eqref{finale} and hereby the proof of Theorem~\ref{Thm2}.

\subsection{Properties of  solutions of the equation $dX_t =\mu(X_t)\,dt + dW_t$}\label{sub23}

Throughout this section we consider the scalar SDE
\begin{equation}\label{generalsde}
dX_t = \mu(X_t)\, dt + dW_t
\end{equation}
and we provide properties of solutions $X$ of~\eqref{generalsde}, which are
used in the proof of Theorem~\ref{Thm1}.

The following lemma provides upper and lower estimates for the probability of 
$X_t$ taking values in bounded intervals.

\begin{lemma}\label{Xpg}
Assume that $\mu$ is measurable and bounded, let $(\Omega,\mathcal F,\PP)$ be a 
complete
probability space, let $W\colon [0,1]\times \Omega\to \R$ be a standard Brownian motion, let $x_0\in\R$ and let $X\colon  [0,1]\times \Omega\to \R$ be a strong solution of the SDE
 \eqref{generalsde} on the time-interval $[0,1]$ with driving Brownian motion $W$ and initial value $x_0$.
Moreover, let 
$\tau\in(0,1]$ and  $M\in(0, \infty)$. 
Then there exist $c_1,c_2\in(0, \infty)$
such that for all $t\in[\tau, 1]$ and all $x,y\in\R$ with $x\leq y$ it holds
\begin{equation}\label{d23}
\PP(X_t\in [x,y] )\leq c_1\, (y-x)    
\end{equation}
and  for all $t\in[\tau, 1]$ and all $x,y\in[-M, M]$ with $x\leq y$ it holds
\begin{equation}\label{d23a}
\PP(X_t\in [x,y])\geq c_2\, (y-x).
\end{equation}
\end{lemma}

\begin{proof}
It is well-known that the assumption that $\mu$ is measurable and bounded implies that for every $t\in (0,1]$, the solution $X_t$  has a Lebesgue density $p_t\colon \R\to [0, \infty)$, which satisfies a two-sided Gaussian bound, i.e. there exist $c_1, c_2, c_3, c_4\in (0, \infty)$ such that for all $t\in(0,1]$ and all $z\in\R$,
\begin{equation}\label{gausdensity}
c_1\cdot \frac{1}{\sqrt{2\pi c_2t}}\cdot e^{-\frac{(z-x_0)^2}{2c_2t}}\leq p_t(z)\leq c_3\cdot \frac{1}{\sqrt{2\pi c_4t}}\cdot e^{-\frac{(z-x_0)^2}{2c_4t}},
\end{equation}
see e.g.~\cite{qz02}.

Let $x,y\in\R$ with $x\leq y$ and  $t\in[\tau, 1]$. Using the second inequality in \eqref{gausdensity} we obtain
\[
\PP(X_t\in [x,y])=\int_x^y p_t(z)\,dz\leq c_3\cdot \frac{1}{\sqrt{2\pi c_4t}}\cdot \int_x^y e^{-\frac{(z-x_0)^2}{2c_4t}}\, dz
\leq c_3\cdot \frac{1}{\sqrt{2\pi c_4\tau}}\cdot (y-x),
\]
which proves the upper bound \eqref{d23}.

 Next assume that $x,y\in[-M, M]$. Employing the first inequality in \eqref{gausdensity} we conclude
\[
\PP(X_t\in [x,y])=\int_x^y p_t(z)\,dz\geq c_1\cdot \frac{1}{\sqrt{2\pi c_2t}}\cdot \int_x^y e^{-\frac{(z-x_0)^2}{2c_2t}}\, dz \geq c_1\cdot \frac{1}{\sqrt{2\pi c_2}}\cdot  e^{-\frac{(M+|x_0|)^2}{2c_2\tau}}\cdot (y-x),
\]
which  proves the lower bound \eqref{d23a} and completes the proof of the lemma.
\end{proof}

Next, we provide 
an estimate for the expected occupation time  of a neighborhood of an arbitrary point $\xi\in\R$ by a strong solution of the SDE~\eqref{generalsde} with deterministic initial value.

\begin{lemma}\label{OT}
Assume that $\mu$ is measurable and bounded, let $(\Omega,\mathcal F,\PP)$ be a 
complete
probability space, 
 let $W\colon [0,1]\times \Omega\to \R$ be a standard Brownian motion, and for every $x\in\R$ let $X^x\colon  [0,1]\times \Omega\to \R$ be a strong solution of the SDE~\eqref{generalsde} on the time-interval $[0,1]$ with driving Brownian motion $W$ and initial value $x$.
Then there exists $c\in(0, \infty)$ such that for all $x,\xi\in\R$, $s\in [0,1]$
 and  all $\varepsilon\in(0, \infty)$, 
\[
\int_{0}^{s}\PP(|X_t^x-\xi|\leq \varepsilon)\, dt \leq c\,\varepsilon\,\sqrt{s}.
\]
\end{lemma}

The proof of Lemma \ref{OT} is similar to the proof of Lemma 4 in~\cite{MGY20}. For convenicence of the reader we present the proof of Lemma \ref{OT} here.

\begin{proof}
Let  $x\in\R$. Clearly,
$X^x$ is a continuous semi-martingale with 
 quadratic variation
\begin{equation}\label{qv}
\langle X^x\rangle_t
      =t,\quad t\in[0,1].
\end{equation}
For $a\in\R$ let $L^a(X^x) = (L^a_t(X^x))_{t\in[0,1]}$ denote the local time 
of $X^x$ at the point $a$.
Hence, 
for all $a\in\R$ and
 all $t\in[0,1]$, 
\begin{align*}
|X_t^x-a| & = |x-a| + \int_0^t \sgn(X_s^x-a)\, \mu (X_s^x)\, ds + \int_0^t \sgn(X_s^x-a)\, dW_s + L^a_t(X^x),
\end{align*}
where $\sgn(z) = 1_{(0,\infty)}(z) - 1_{(-\infty,0]}(z)$ for $z\in\R$,
see, e.g.~\cite[Chap. VI]{RevuzYor2005}.
Thus, 
for all $a\in\R$ and
 all $t\in[0,1]$,
\begin{align*}
L^a_t(X^x) & \le |X^x_t-x| + \int_0^t |\mu (X_s^x)|\, ds + \Bigl|\int_0^t \sgn(X_s^x-a)\, dW_s\Bigr|\\
& \le 2\,\int_0^t |\mu (X_s^x)|\, ds + |W_t| + \Bigl|\int_0^t \sgn(X_s^x-a)\, dW_s\Bigr|.
\end{align*}

Since $\mu$ is bounded
we may conclude that 
there exists $c\in (0,\infty)$ such that
  for all $x\in\R$,
all $a\in\R$ 
   and all $t\in[0,1]$,
\begin{equation}\label{local1}
\EE\bigl[L^a_t(X^x)\bigr]  \le  c\,\sqrt{t}.
\end{equation}
Using~\eqref{qv} and~\eqref{local1} we obtain by the occupation 
time
formula that
for all $x,\xi\in\R$, all $s\in [0,1]$ and all $\eps\in (0,\infty)$,
\[
  \int_{0}^{s}\PP(|X_t^x-\xi|\leq \varepsilon)\, dt=\EE\Bigl[\int_0^s 1_{[\xi-\eps,\xi+\eps]}(X^x_{t})\, dt\Bigr] = \int_{\R}1_{[\xi-\eps,\xi+\eps]}(a)\, \EE\bigl[L^a_s(X^x)\bigr]\, da 
 \le 2c\,  \eps\,\sqrt{s},
\]
which completes the proof of the lemma.
\end{proof}

The following lemma provides under the condition ($\mu1$) a common functional representation of arbitrary solutions of the SDE~\eqref{generalsde} as well as the transition probabilities and a comparison result for strong solutions of~\eqref{generalsde} with deterministic initial values.

\begin{lemma}\label{markov}
Assume that $\mu$ satisfies ($\mu1$). Then strong existence and pathwise uniqueness hold for the SDE~\eqref{generalsde}. In particular,  for every $T\in (0,\infty)$ there exists a Borel-measurable  mapping 
\[
F\colon \R\times C([0,T],\R)\to C([0,T],\R)
\]
such that for every 
complete
probability space $(\Omega,\mathcal F,\PP)$, every standard Brownian motion $W\colon [0,T]\times \Omega\to \R$ and every random variable $\eta\colon\Omega\to \R$ such that $W$ and $\eta$ are independent it holds that
\begin{itemize}
\item[(i)] if   $X\colon[0,T]\times \Omega\to \R$ is a  solution of the SDE~\eqref{generalsde} on the time-interval $[0,T]$ with driving Brownian motion $W$ and initial value $\eta$ then $X=F(\eta,W)$ $\PP$-almost surely,
\item[(ii)] $F(\eta,W)$ is a strong solution of the SDE~\eqref{generalsde} on the time-interval $[0,T]$ with driving Brownian motion $W$ and initial value $\eta$.
\end{itemize}

Moreover, let $(\Omega,\mathcal F,\PP)$ be a
complete
 probability space, 
let $T\in (0,\infty)$,
let $W\colon [0,T]\times \Omega\to \R$ be a standard Brownian motion, and for every $x\in\R$ let $X^x\colon  [0,T]\times \Omega\to \R$ be a strong solution of the SDE~\eqref{generalsde} on the time-interval $[0,T]$ with driving Brownian motion $W$ and initial value $x$. Then 
\begin{itemize}
\item[(iii)] for all $s\in[0,T]$ and $\PP ^{X_s} $-almost all $x\in\R$ we have
\[
\PP^{(X^x_t)_{t\in [s,T]}|X^x_s=y}=\PP^{(X_t^y)_{t\in [0,T-s]}},
\]
\item[(iv)] for all $x,y\in \R$ with $x\le y$ we have
\[
\PP(\forall\, t\in[0,T]\colon X^x_t \le X^y_t) = 1.
\]
\end{itemize}
\end{lemma}

\begin{proof}
As a straightforward generalization of Lemma 7 and Lemma 8 in~\cite{MGY20} one obtains that
there exist Lipschitz continuous functions  $\widetilde\mu, \widetilde\sigma\colon\R\to\R$
 and a strictly increasing, 
Lipschitz continuous bijection $G\colon\R\to\R$ 
with a Lipschitz continuous inverse $G^{-1}\colon\R\to\R$  such that for every $T\in (0,\infty)$, every 
complete
probability space $(\Omega,\mathcal F,\PP)$, every standard Brownian motion $W\colon [0,T]\times \Omega\to \R$ and every random variable $\eta\colon\Omega\to \R$ such that $W$ and $\eta$ are independent it holds that 
\begin{itemize}
\item[a)] if $X\colon [0,T]\times \Omega\to \R$ is a (strong) solution of the SDE
\begin{equation}\label{generalsdetrans}
dX_t = \widetilde\mu(X_t)\, dt + \widetilde\sigma(X_t)\, dW_t
\end{equation}
on the time-interval $[0,T]$ with driving Brownian motion $W$ and initial value $\eta$ then $G^{-1}\circ X$ is a (strong) solution of the SDE~\eqref{generalsde} on the time-interval $[0,T]$ with driving Brownian motion $W$ and initial value $G^{-1} (\eta)$,
\item[b)] if $X\colon [0,T]\times \Omega\to \R$ is a (strong) solution of the SDE~\eqref{generalsde}  on the time-interval $[0,T]$ with driving Brownian motion $W$ and initial value $\eta$ then $G\circ X$ is a (strong) solution of the SDE~\eqref{generalsdetrans} on the time-interval $[0,T]$ 
with driving Brownian motion $W$ and initial value $G(\eta)$.
\end{itemize}

By the Lipschitz continuity of $\widetilde \mu$ and $ \widetilde\sigma$ strong existence and pathwise uniqueness hold for the SDE~\eqref{generalsdetrans}. Using a) and b) it follows that strong existence and pathwise uniqueness hold for the SDE~\eqref{generalsde} as well. For every $T\in (0,\infty)$ the existence of $F$ with the property (i) is now a consequence of Theorem 1 in~\cite{Ka96}. Strong existence for the SDE~\eqref{generalsde} and property (i) jointly imply that $F$  has property (ii) as well.

We turn to the proof of (iii). Let $s\in [0,T]$ and choose a Borel measurable $F\colon\R\times C([0,T-s],\R)\to C([0,T-s],\R) $ according to the already proven part of the lemma. In particular, for all $x\in \R$ we have  $(X^x_t)_{t\in [0,T-s]} = F(x,(W_t)_{t\in [0,T-s]})$ almost surely. Let $x\in \R$. The process $(X^x_{s+t})_{t\in [0,T-s]}$  is a solution of the SDE~\eqref{generalsde} on the time-interval $[0,T-s]$ with driving Brownian motion $(W_{s+t}-W_s)_{t\in [0,T-s]}$ and initial value $X^x_s$. By property (i) of $F$ we thus have $(X^x_{s+t})_{t\in [0,T-s]} = F(X^x_s,(W_{s+t}-W_s)_{t\in [0,T-s]}))$ almost surely. It follows that for $\PP ^{X^x_s} $-almost all $y\in\R$ we have
\[ 
\PP^{(X^x_{s+t})_{t\in [0,T-s]}|X^x_s=y} = \PP^{F(y, (W_{s+t}-W_s)_{t\in [0,T-s]})} = \PP^{F(y,(W_t)_{t\in [0,T-s]})} = \PP^{(X^y_t)_{t\in [0,T-s]}}.
\]

Finally, we prove (iv). Let $x,y\in\R$ with $x\le y$. 
Then 
$G(x)\le G(y)$. Using b) and a comparison result for SDEs with Lipschitz continuous coefficients, e.g.~\cite[Proposition 5.2.18]{ks91}, we thus obtain that
$\PP( \forall\, t\in[0,T]\colon G(X^x_t) \le G(X^y_t)) = 1$. Since $G^{-1}$ is increasing we furthermore have $\{\forall\, t\in[0,T]\colon G(X^x_t) \le G(X^y_t)\}\subset \{\forall\, t\in[0,T]\colon X^x_t \le X^y_t\}$, which finishes the proof of (iv).

This completes the proof of the lemma.
\end{proof}

Finally, we  introduce  a simple approximation of strong solutions of the SDE~\eqref{generalsde} on a time-interval $[s,t]$ and provide an estimate of the mean squared total amount of time the solution and its approximation lie on different sides of a fixed horizontal line.

\begin{lemma}\label{Yproc}
Assume that $\mu$ satisfies ($\mu1$) and ($\mu5$). Let $(\Omega,\mathcal F,\PP)$ be a 
complete
probability space, let $W\colon [0,1]\times \Omega\to \R$ be a standard Brownian motion, and for every $x\in\R$ let $X^x\colon  [0,1]\times \Omega\to \R$ be a strong solution of the SDE~\eqref{generalsde} on the time-interval $[0,1]$ with driving Brownian motion $W$ and initial value $x$. 
Then there exists $c\in(0, \infty)$ such that for all $\xi\in\R$, all $s,t\in [0,1]$ with $s<t$ and all $x\in\R$,
\begin{equation}\label{ineq1}
\EE\Bigl[\Bigl(\int_{s}^{t}1_{\{(X^x_u-\xi)\, (X^x_s + W_u-W_s-\xi)\leq 0\}} \, du\Bigr)^2\Bigr]\leq c\,(t-s)^3.
\end{equation}
\end{lemma}
\begin{proof}
Fix $\xi\in \R$, $s,t\in [0,1]$ with $s<t$ and $x\in \R$, and put
\[
Y_u = X^x_s + W_u-W_s
\]
for $u\in[s,t]$. By 
($\mu5$) we have for all $u\in [s,t]$,
\[
|X^x_u-Y_u|=\Bigl|\int_s^u \mu(X_v)\, dv\Bigr|
\leq \|\mu\|_\infty\,(u-s).
\]
Hence,
\begin{align*}
\Bigl(\int_{s}^{t}1_{\{(X^x_u-\xi)\, (Y_u-\xi)\leq 0\}} \, du\Bigr)^2
&=\int_{s}^{t}\int_{s}^{t} 1_{\{(X^x_u-\xi)\, (Y_{u}-\xi)\leq 0\}}\, 1_{\{(X^x_v-\xi)\, (Y_{v}-\xi)\leq 0\}} \, du\, dv\\
&\leq \int_{s}^{t}\int_{s}^{t}  1_{\{|X^x_u-\xi|\leq |X^x_u-Y_{u}|\}}\, 1_{\{|X^x_v-\xi|\leq |X^x_v-Y_{v}|\}} \, du\, dv\\
&\leq \int_{s}^{t}\int_{s}^{t} 1_{\{|X^x_u-\xi|\leq \|\mu\|_\infty\, |t-s|\}}\, 1_{\{|X^x_v-\xi|\leq \|\mu\|_\infty\, |t-s|\}} \, du\, dv\\
&= 2 \int_{s}^{t}\int_{v}^{t}  1_{\{|X^x_u-\xi|\leq \|\mu\|_\infty\, |t-s|\}}\, 1_{\{|X^x_v-\xi|\leq \|\mu\|_\infty\, |t-s|\}} \, du\, dv,
\end{align*}
and therefore,
\begin{equation}\label{z12}
\begin{aligned}
& \EE\Bigl[\Bigl(\int_{s}^{t}1_{\{(X^x_u-\xi)\, (Y_{u}-\xi)\leq 0\}} \, du\Bigr)^2\Bigr]\\
& \qquad\qquad \leq 2 \int_{s}^{t}\EE\Bigl[ 1_{\{|X^x_v-\xi|\leq \|\mu\|_\infty\,|t-s|\}}\,\EE\Bigl[\int_{v}^{1}  1_{\{|X^x_u-\xi|\leq \|\mu\|_\infty\, |t-s|\}}  \, du\Bigr|X^x_v\Bigr]\Bigr]\, dv.
\end{aligned}
\end{equation}
Using Lemma \ref{markov}(iii) and then Lemma  \ref{OT} we obtain 
that there exists $c_1\in(0, \infty)$, which only depends on $\mu$, such that for all $v\in[s, t]$ and $\PP^{X^x_v}$-almost all $y\in\R$,
\begin{equation}\label{z13}
\begin{aligned}
& \EE\Bigl[\int_{v}^{t}  1_{\{|X^x_u-\xi|\leq \|\mu\|_\infty\, |t-s|\}}  \, du\Bigr|X^x_v=y\Bigr]\\
&\qquad\qquad =\EE\Bigl[\int_{v}^{t}  1_{\{|X^{y}_{u-v}-\xi|\leq \|\mu\|_\infty\, |t-s|\}}  \, du\Bigr]\\
&\qquad\qquad
=\int_{0}^{t-v}  \PP(|X^{y}_{u}-\xi|\leq \|\mu\|_\infty\, |t-s|) \, du \leq c_1\, |t-s|^{3/2}.      
\end{aligned}
\end{equation}
   Combining~\eqref{z12} and~\eqref{z13} we obtain
   \begin{equation}\label{eqz1}
\begin{aligned}
&\EE\Bigl[\Bigl(\int_{s}^{t}1_{\{(X^x_u-\xi)\, (Y_{u}-\xi)\leq 0\}} \, du\Bigr)^2\Bigr]\\
& \qquad \qquad \leq 2c_1 \, |t-s|^{3/2}\, \int_{s}^{t}\PP(|X^x_v-\xi|\leq \|\mu\|_\infty\, |t-s|) \,dv\\
& \qquad \qquad =2c_1 \, |t-s|^{3/2}\,\EE\Bigl[\EE\Bigl[\int_{s}^{t}1_{\{|X^x_v-\xi|\leq \|\mu\|_\infty\, |t-s|\}}\, dv\Bigl| X^x_s\Bigr]\Bigr].
\end{aligned}
\end{equation}
Using again Lemma \ref{markov}(iii) and then Lemma  \ref{OT} we conclude similar to~\eqref{z13} that there exists $c_2\in(0, \infty)$, which only depends on $\mu$, such that for  $\PP^{X^x_s}$-almost all $y\in\R$,
\begin{equation}\label{eqz2}
 \EE\Bigl[\int_{s}^{t}  1_{\{|X^x_v-\xi|\leq \|\mu\|_\infty\, |t-s|\}}  \, du\Bigr|X^x_s=y\Bigr]
=\int_{0}^{t-s}  \PP(|X^{y}_{v}-\xi|\leq \|\mu\|_\infty\, |t-s|) \, dv \leq c_2\, |t-s|^{3/2}.      
\end{equation}
Combining~\eqref{eqz1} and~\eqref{eqz2} completes the proof of  the lemma.
\end{proof}

\subsection{Proof of Theorem \ref{Thm1}}\label{sub24}

In the following, let $(\Omega,\F,\PP)$ be a complete probability space and let $W\colon [0,1]\times \Omega\to \R$ be a standard Brownian motion. We will prove the following proposition, which clearly implies Theorem~\ref{Thm1}.

\begin{prop}\label{unprop}
 Assume that $\mu$ satisfies ($\mu1$) to ($\mu5$), let $x_0\in\R$ and let $X\colon [0,1]\times \Omega\to \R$ be a strong solution of the SDE~\eqref{generalsde} on the time-interval $[0,1]$ with initial value $x_0$ and driving Brownian motion $W$. Then  
 there exist $c_1,c_2\in (0,\infty)$ such that for all $n\in 2\N$, all $t_1, \ldots, t_n\in [0,1]$ with
 \begin{equation}\label{disci}
0<t_1<\ldots<t_n=1
\end{equation}
and
 \begin{equation}\label{b5i}
 2/n,4/n,\dots,1\in \{t_1,\dots,t_n\},
 \end{equation}
 and all measurable $g\colon \R^n\to \R$ we have
\begin{equation}\label{end555} 
\EE[|X_1-g(W_{t_1},\dots,W_{t_n}|]\ge \frac{c_1}{n^{3/4}}\,(\max(0,(1-c_2/n^{1/16})))^{3/2}.
\end{equation}
\end{prop}

For the proof of Proposition~\ref{unprop} we fix $n\in 2\N$ as well as $t_1, \ldots, t_n\in [0,1]$ with~\eqref{disci} and~\eqref{b5i}. Moreover, we put $t_0=0$.

Throughout this section we use $c,c_1,c_2\in(0,\infty)$ to denote positive constants that may change their values in every appearance but neither depend on $n$ nor on the time-points  $t_1, \ldots, t_{n}$.

Recall from Section~\ref{sub22} the definition and the properties of the processes $\overline W, B, \widetilde B,\widetilde W\colon [0,1]\times \Omega\to \R$ associated with the discretization~\eqref{disci}. 
In particular, $\widetilde W$ is a Brownian motion and for all $i\in\{0, \ldots, n\}$ we have
\begin{equation}\label{Brm}
W_{t_i}=\widetilde W_{t_i}.
\end{equation}

\begin{lemma}\label{lemmanew02} Assume that $\mu$ satisfies ($\mu1$), let $x_0\in\R$ and let $X,\widetilde X\colon [0,1]\times \Omega\to \R$ be strong solutions of the SDE~\eqref{generalsde} on the time-interval $[0,1]$ with initial value $x_0$ and driving Brownian motion $W$ and $\widetilde W$, respectively. Then for all measurable $g\colon \R^n\to \R$ and all $p\in[1,\infty)$,
\begin{equation}\label{gL1}
\begin{aligned}
\bigl(\EE[|X_1-g(W_{t_1}, \ldots, W_{t_n})|^p]\bigr)^{1/p}  \geq \frac{1}{2}\,   \bigl(\EE[|X_1-\widetilde X_1|^p]\bigr)^{1/p}.
\end{aligned}
\end{equation}
\end{lemma}

\begin{proof}
By Lemma~\ref{markov}(i) there exists a measurable function $F\colon \R\times C([0,1],\R)\to C([0,1],\R)$ such that $\PP$-almost surely,
\begin{equation}\label{sol}
X = F(x_0,W)\, \text{ and }\,\widetilde X =F(x_0,\widetilde W).
\end{equation}
Hence there exist measurable functions $\Phi\colon C([0,1],\R)\times C([0,1],\R)\to\R$ and $\varphi\colon C([0,1],\R)\to\R$ such that $\PP$-almost surely
\begin{equation}\label{repr}
X_1=\Phi(\overline W, B),\,\, \widetilde X_1 = \Phi(\overline W, \widetilde B),\,\, g(W_{t_1}, \ldots, W_{t_{n}})=\varphi(\overline W).
\end{equation}

Since $\PP^{(\overline W, B)}=\PP^{(\overline W, \widetilde B)}$, see~\eqref{eqdis}, we may apply
Lemma \ref{symm} with $\Omega_1=\Omega_2=C([0,1],\R)$, $V_1=\overline W$, $V_2=B$, $V_2'=\widetilde B$ and $\Phi, \varphi$ as in~\eqref{repr} to obtain~\eqref{gL1}.
\end{proof}

In the analysis of the right hand side of
\eqref{gL1} we will make use of the following upper bound on the $L_p$-distance between the two processes $X$ and $\widetilde X$ at the time points $t_0, \ldots, t_n$.
\begin{lemma}\label{lab2}
Assume that $\mu$ satisfies ($\mu1$), ($\mu2$) and ($\mu5$), let $x_0\in\R$ and let $X,\widetilde X\colon [0,1]\times \Omega\to \R$ be strong solutions of the SDE~\eqref{generalsde} on the time-interval $[0,1]$ with initial value $x_0$ and driving Brownian motion $W$ and $\widetilde W$, respectively. Then
for every $p\in[1, \infty)$
there exists $c\in(0, \infty)$ such that
\[
\max_{i\in\{0, \ldots, n\}}\EE\bigl[|X_{t_i}-\widetilde X_{t_i}|^p\bigr]^{1/p} \leq \frac{c}{n^{3/4}}.
\]
\end{lemma}
\begin{proof}
Let $i\in\{0, \ldots, n\}$. We have
\begin{equation}\label{pp4}
\EE\bigl[|X_{t_i}-\widetilde X_{t_i}|^p\bigr]^{1/p} \leq \EE\bigl[|X_{\underline{t_i}}-\widetilde X_{\underline{t_i}}|^p\bigr]^{1/p}+\EE\bigl[|X_{t_i}- X_{\underline {t_i}}-\widetilde X_{t_i}+\widetilde X_{\underline{t_i}}|^p\bigr]^{1/p},
\end{equation}
where
\[
\underline {t_i}=\max\{\tau\in\{2j/n\colon j=0, \ldots, n/2\}\colon t_j\geq \tau\}.
\]
Observing the fact that 
\[
t_i-\underline {t_i}\leq \frac{2}{n}
\]
 and employing \eqref{Brm} we obtain 
 that 
\begin{align*}
|X_{t_i}- X_{\underline {t_i}}-\widetilde X_{t_i}+\widetilde X_{\underline{t_i}}|&=\Bigl|\int_{\underline{t_i}}^{t_i}(\mu(X_s)-\mu(\widetilde X_s))\, ds\Bigr|
\leq \frac{4 \|\mu\|_{\infty}}{n}.
\end{align*}
Observing  ($\mu5$) we thus see that there exists $c\in (0,\infty)$ such that
\begin{equation}\label{pp3}
\max_{i\in\{0,\dots,n\}}\EE\bigl[|X_{t_i}- X_{\underline {t_i}}-\widetilde X_{t_i}+\widetilde X_{\underline{t_i}}|^p\bigr]^{1/p}\leq \frac{c}{n}.
\end{equation}

Put $k_n = n/2$ and let $ Y_{k_n}=(Y_{k_n, t})_{t\in[0,1]}$ and 
$\widetilde Y_{k_n}=(\widetilde Y_{k_n, t})_{t\in[0,1]}$ denote the transformed time-continuous quasi-Milstein schemes from~\cite[Section 4]{MGY19b} that have step-size $1/k_n$ and are associated to the SDE~\eqref{generalsde} on the time-interval $[0,1]$ with initial value $x_0$ and driving Brownian motion $W$ and $\widetilde W$, respectively.  Since $\mu$ satisfies ($\mu1$) and ($\mu2$) we may apply  Theorem 4 in~\cite{MGY19b} to obtain that  
\begin{equation}\label{pp1}
\max_{j\in\{0, \ldots, k_n\}}\Bigl(\EE\bigl[|X_{j/k_n}-Y_{k_n, j/k_n}|^p\bigr]^{1/p} 
+ \EE\bigl[|\widetilde X_{j/k_n}-\widetilde Y_{k_n, j/k_n}|^p\bigr]^{1/p}\Bigr)\leq \frac{c}{k_n^{3/4}}.
\end{equation}
By the definition of $Y_{k_n}$ and $\widetilde Y_{k_n}$  we have for every $j\in\{0, \ldots, k_n\}$,
\begin{equation}\label{tm}
Y_{k_n, j/k_n}=g_j( W_{1/k_n}, \ldots, W_{j/k_n}),
\quad \widetilde Y_{k_n, j/k_n}=g_j( \widetilde W_{1/k_n}, \ldots,\widetilde W_{j/k_n})
\end{equation}
with a
measurable $g_j\colon\R^j\to\R$. Observing \eqref{Brm} we conclude 
by~\eqref{pp1} and~\eqref{tm} that
\begin{align*}
\max_{i\in\{0,\dots,n\}} \EE\bigl[|X_{\underline{t_i}}-\widetilde X_{\underline{t_i}}|^p\bigr]^{1/p} & \leq \max_{i\in\{0,\dots,n\}}\Bigl(\EE\bigl[|X_{\underline{t_i}}-Y_{k_n, \underline{t_i}}|^p\bigr]^{1/p}+\EE\bigl[|\widetilde X_{\underline{t_i}}-Y_{k_n, \underline{t_i}}|^p\bigr]^{1/p}\Bigr)\\
& = \max_{i\in\{0,\dots,n\}}\Bigl(\EE\bigl[|X_{\underline{t_i}}-Y_{k_n, \underline{t_i}}|^p\bigr]^{1/p}+\EE\bigl[|\widetilde X_{\underline{t_i}}-\widetilde Y_{k_n, \underline{t_i}}|^p\bigr]^{1/p}\Bigr)\\
& \leq \frac{c}{n^{3/4}}.
\end{align*}
Combining the latter estimate with \eqref{pp4} and \eqref{pp3} yields the statement of the lemma.
\end{proof}

The following three lemmas are crucial to obtain a lower bound for the right hand side of \eqref{gL1} in the case $p=2$. 

\begin{lemma}\label{mixed}
Assume that $\mu$ satisfies ($\mu1$) and ($\mu4$). Let $x_0\in\R$ and let $X,\widetilde X\colon [0,1]\times \Omega\to \R$ be strong solutions of the SDE~\eqref{generalsde} on the time-interval $[0,1]$ with initial value $x_0$ and driving Brownian motion $W$ and $\widetilde W$, respectively. Then for all $i\in\{1,\dots,n\}$ we have
\begin{equation}\label{neweqq1}
\EE\Bigl[(X_{t_{i-1}}-\widetilde X_{t_{i-1}})\,\int_{t_{i-1}}^{t_i}(\mu(X_s)-\mu(\widetilde X_s)) \, ds\Bigr] \ge 0.
\end{equation}
\end{lemma}

\begin{proof}
Fix $i\in\{1,\dots,n\}$ and choose $F\colon\R\times C([0,t_i-t_{i-1}],\R) \to C([0,t_i-t_{i-1}],\R)$ according to Lemma~\ref{markov}. Put
\[
V=(V_t= W_{t_{i-1}+t}-W_{t_{i-1}})_{t\in[0,t_i-t_{i-1}]}\text{ and }\widetilde V=(\widetilde V_t= \widetilde W_{t_{i-1}+t}-\widetilde W_{t_{i-1}})_{t\in[0,t_i-t_{i-1}]}.
\]
Since the processes $(X_{t_{i-1}+t})_{t\in[0,t_i-t_{i-1}]}$ and $(\widetilde X_{t_{i-1}+t})_{t\in[0,t_i-t_{i-1}]}$ are solutions of the SDE~\eqref{generalsde} 
on the time interval $[t_{i-1}, t_i]$
with initial value $X_{t_{i-1}}$   and driving Brownian motion $V$ and with initial value $\widetilde X_{t_{i-1}}$ and driving Brownian motion $\widetilde V$, respectively, we know by Lemma~\ref{markov}(i) that $\PP$-almost surely
\begin{equation}\label{bbb3}
\begin{aligned}
(X_{t_{i-1}+t})_{t\in[0,t_i-t_{i-1}]} & =F(X_{t_{i-1}},V),\\
(\widetilde X_{t_{i-1}+t})_{t\in[0,t_i-t_{i-1}]} & =F(\widetilde X_{t_{i-1}},\widetilde V).
\end{aligned}
\end{equation}

Note that the random vector $(X_{t_{i-1}},\widetilde X_{t_{i-1}})$ is $\mathcal G/\mathcal B(\R^2)$-measurable, where $\mathcal B(\R^2)$ is the Borel $\sigma$-field in $\R^2$ and $\mathcal G\subset \mathcal F$ is the completion of the $\sigma$-field generated by $( W_t,\widetilde W_t)_{t\in [0,t_{i-1}]}$, i.e.
\[
\mathcal G = \sigma\bigl(\sigma(\{( W_t,\widetilde W_t)\colon t\in [0,t_{i-1}]\})\cup \mathcal N\bigr),
\]
 where $\mathcal N = \{N\in \mathcal F\colon \PP(N)=0\}$. 
 Furthermore, by definition of $\widetilde W$ we have
\begin{align*}
( W_t,\widetilde W_t)_{t\in [0,t_{i-1}]}& = \psi(( W_t,\widetilde B_t)_{ t\in [0,t_{i-1}]}),\\
(V_{t}, \widetilde V_{t})_{t\in [0,t_i-t_{i-1}]} &= 
\varphi( (W_t-W_{t_{i-1}},\widetilde B_t)_{ t\in [t_{i-1},t_i]})
\end{align*}
for some measurable mappings $\psi\colon C([0,t_{i-1}],\R)^2 \to C([0,t_{i-1}],\R)^2$ and $\varphi\colon C([t_{i-1},t_i],\R)^2 \to C([0,t_i-t_{i-1}],\R)^2$,
which implies that $( W_t,\widetilde W_t)_{t\in [0,t_{i-1}]}$ and $(V_{t}, \widetilde V_{t})_{t\in [0,t_i-t_{i-1}]}$ are independent. As a consequence, the $\sigma$-fields $\mathcal G$ and $\sigma(V,\widetilde V)$ are independent, which in turn implies the independence of $(X_{t_{i-1}},\widetilde X_{t_{i-1}})$ and $(V,\widetilde V)$.

Using~\eqref{bbb3} we thus have for $\PP^{(X_{t_{i-1}},\widetilde X_{t_{i-1}})}$-almost all $(y,\tilde y)\in \R^2$ that
\begin{equation}\label{ccc1}
\begin{aligned}
& \EE\Bigl[(X_{t_{i-1}}-\widetilde X_{t_{i-1}})\,\int_{t_{i-1}}^{t_i}(\mu(X_s)-\mu(\widetilde X_s)) \, ds\Bigl|(X_{t_{i-1}},\widetilde X_{t_{i-1}})=(y,\tilde y)\Bigr]\\
& \qquad\qquad = (y-\tilde y)\, \EE\Bigl[\int_{t_{i-1}}^{t_i} \bigl(\mu(F(y,V)(s)) - \mu(F(\tilde y,\widetilde V)(s)) \bigr) \, ds \Bigr]\\
& \qquad\qquad = (y-\tilde y)\, \EE\Bigl[\int_{t_{i-1}}^{t_i} \bigl(\mu(F(y,V)(s)) - \mu(F(\tilde y,V)(s)) \bigr) \, ds \Bigr].
\end{aligned}
\end{equation}

By Lemma~\ref{markov}(ii) we know that $F(y,V)$ and $F(\tilde y,V)$ are strong solutions of the SDE~\eqref{generalsde} on the time-interval $[0,t_i-t_{i-1}]$ with driving Brownian motion $V$ and initial value $y$ and $\tilde y$, respectively. Using Lemma~\ref{markov}(iv) and the assumption that $\mu$ is increasing we conclude that $\PP$-almost surely
\[
\forall\, s\in [t_{i-1},t_i]\colon \quad (y-\tilde y)\,  \bigl(\mu(F(y,V)(s)) - \mu(F(\tilde y,V)(s))\bigr) \ge 0.
\]
Combining the latter fact with~\eqref{ccc1} we conclude that  for $\PP^{(X_{t_{i-1}},\widetilde X_{t_{i-1}})}$-almost all $(y,\tilde y)\in \R^2$
\begin{equation}\label{cvcv1}
\EE\Bigl[(X_{t_{i-1}}-\widetilde X_{t_{i-1}})\,\int_{t_{i-1}}^{t_i}(\mu(X_s)-\mu(\widetilde X_s)) \, ds\Bigl|(X_{t_{i-1}},\widetilde X_{t_{i-1}})=(y,\tilde y)\Bigr] \ge 0,
\end{equation}
which clearly implies~\eqref{neweqq1}. 
\end{proof}

\begin{lemma}\label{diagonal1} Assume that $\mu$ satisfies ($\mu1$) and ($\mu5$). Let $x_0\in\R$ and let $X,\widetilde X\colon [0,1]\times \Omega\to \R$ be strong solutions of the SDE~\eqref{generalsde} on the time-interval $[0,1]$ with initial value $x_0$ and driving Brownian motion $W$ and $\widetilde W$, respectively. Then 
there exists $c\in (0,\infty)$ such that for all $i\in\{1,\dots,n\}$ with $t_{i}> 1/2$ it holds  
\begin{equation}\label{i3}
\begin{aligned}
& \EE\Bigl[\Bigl(\int_{t_{i-1}}^{t_i} \bigl(\mu(X_s)-\mu(\widetilde X_s)\bigr)\, ds\Bigr)^2\Bigr]\\ & \qquad\qquad \ge \frac{1}{4}
\EE\Bigl[\Bigr(\int_{t_{i-1}}^{t_i} \bigl(\mu(X_{t_{i-1}} + W_s-W_{t_{i-1}})-\mu( X_{t_{i-1}} + \widetilde W_s-\widetilde W_{t_{i-1}})\bigr)\, ds\Bigr)^2\Bigr] \\
& \qquad\qquad\qquad\qquad  -\frac{c}{n^{5/2+1/16}}.
\end{aligned}
\end{equation}
\end{lemma}

\begin{proof} 
For all $i\in  \{1,\dots,n\}$  put
\begin{align*}
A_i & = \int_{t_{i-1}}^{t_i}\bigl( \mu(X_{t_{i-1}} + W_t-W_{t_{i-1}})-\mu( X_{t_{i-1}} + \widetilde W_t-\widetilde W_{t_{i-1}})\bigr)\, dt,\\
B_i & =  \int_{t_{i-1}}^{t_i}\bigl(\mu(X_{t_{i-1}} + W_t-W_{t_{i-1}})-  \mu(X_t)\bigr)\, dt,\\
C_i & =  \int_{t_{i-1}}^{t_i}\bigl(\mu(X_t) -\mu(\widetilde X_t)\bigr)\, dt,\\
D_i & =   \int_{t_{i-1}}^{t_i}\bigl(\mu(\widetilde X_t) - \mu( \widetilde X_{t_{i-1}} + \widetilde W_t-\widetilde W_{t_{i-1}})\bigr)\, dt,\\
E_i & = \int_{t_{i-1}}^{t_i} \bigl(\mu( \widetilde X_{t_{i-1}} + \widetilde W_t-\widetilde W_{t_{i-1}}) - \mu( X_{t_{i-1}} + \widetilde W_t-\widetilde W_{t_{i-1}})\bigr)\, dt.
\end{align*} 
Clearly, $A_i = B_i + C_i + D_i + E_i$, which yields 
\begin{equation}\label{nbv2}
\EE[A_i^2] \le 4(\EE[B_i^2] + \EE[C_i^2] + \EE[D_i^2] + \EE[E_i^2]) = 4(2\EE[B_i^2] + \EE[C_i^2]  + \EE[E_i^2]).
\end{equation}

By Lemma~\ref{basics} we see that 
 there exists $c\in(0, \infty)$ such that for all $i\in\{1,\dots,n\}$ and all $t\in [t_{i-1},t_i]$,
\begin{align*}
& |\mu(X_{t_{i-1}}+W_t-W_{t_{i-1}})-\mu(X_t)| \\
& \qquad\qquad \leq  c\, \bigl(|X_{t_{i-1}}+W_t-W_{t_{i-1}}-X_t|+\sum_{j=1}^k 1_{\{(X_{t_{i-1}}+W_t-W_{t_{i-1}}-\xi_j)\, (X_t-\xi_j)\leq 0\}}\bigr).
\end{align*}
For all $i\in\{1,\dots,n\}$ and all $t\in [t_{i-1},t_i]$ we furthermore have
\[
|X_{t_{i-1}}+W_t-W_{t_{i-1}}-X_t|=\Bigl|\int_{t_{i-1}}^{t} \mu(X_s)\, ds\Bigr|\leq \|\mu\|_{\infty} \, (t-t_{i-1}).
\]
Employing ($\mu5$), Lemma~\ref{Yproc} and~\eqref{b5i} we thus obtain that  there exist $c_1, c_2,c_3\in(0, \infty)$ such that
for all $i\in\{1,\dots,n\}$,
\begin{equation}\label{uno}
\begin{aligned}
\EE[B_i^2]&\leq c_1 \,  \EE\Bigl[\Bigl((t_i-t_{i-1})^2+\sum_{j=1}^k \int_{t_{i-1}}^{t_i}1_{\{(X_{t_{i-1}}+W_t-W_{t_{i-1}}-\xi_j)\, (X_t-\xi_j)\leq 0\}} \, du\Bigr)^2\Bigr]\\
&\leq c_2\, \Bigl((t_i-t_{i-1})^4+\sum_{j=1}^k \EE\Bigl[\Bigl(\int_{t_{i-1}}^{t_i}1_{\{(X_{t_{i-1}}+W_t-W_{t_{i-1}}-\xi_j)\, (X_t-\xi_j)\leq 0\}} \, du\Bigr)^2\Bigr]\Bigr)\\
&\leq c_3\, \bigl((t_i-t_{i-1})^4+(t_i-t_{i-1})^3) \le \frac{16c_3}{n^3}.
\end{aligned}
\end{equation}

Clearly, for all $i\in  \{1,\dots,n\}$, 
\[
E_i^2 \le 2\|\mu\|_\infty\,(t_i-t_{i-1})\,\int_{t_{i-1}}^{t_i} \bigl|\mu( \widetilde X_{t_{i-1}} + \widetilde W_s-\widetilde W_{t_{i-1}}) - \mu( X_{t_{i-1}} + \widetilde W_s-\widetilde W_{t_{i-1}})\bigr|\, dt,
\] 
and therefore
\begin{equation}\label{b1}
\EE[E_i^2] 
  \leq 2\|\mu\|_{\infty}\, (t_i-t_{i-1})\,\int_{t_{i-1}}^{t_i}\EE[|\mu(X_{t_{i-1}}+\widetilde W_t-\widetilde W_{t_{i-1}})-\mu(\widetilde X_{t_{i-1}}+\widetilde W_t-\widetilde W_{t_{i-1}})|] \, dt.
 \end{equation}
Recall from the proof of Lemma~\ref{mixed} that for all $i\in  \{1,\dots,n\}$, $(X_{t_{i-1}},\widetilde X_{t_{i-1}})$ and $(\widetilde W_t - \widetilde W_{t_{i-1}})_{t\in [t_{i-1},t_i]}$ are independent. Hence, for all $i\in  \{1,\dots,n\}$,
\begin{equation}\label{tmg44}
\begin{aligned}
& \int_{t_{i-1}}^{t_i}\EE[|\mu(X_{t_{i-1}}+\widetilde W_t-\widetilde W_{t_{i-1}})-\mu(\widetilde X_{t_{i-1}}+\widetilde W_t-\widetilde W_{t_{i-1}})|] \, dt \\
& \qquad\qquad =  \int_{t_{i-1}}^{t_i}\int_{\R}\EE[|\mu(X_{t_{i-1}}+u)-\mu(\widetilde X_{t_{i-1}}+u)|]\, \PP^{\widetilde W_t-\widetilde W_{t_{i-1}}} (du)\, dt. 
\end{aligned}
\end{equation}
By Lemma~\ref{basics} we know that there exists $c\in(0, \infty)$ such that for
$i\in\{1, \ldots, n\}$ and 
all $u\in\R$,
\begin{equation}\label{b2}
\begin{aligned}
& |\mu(X_{t_{i-1}}+u)-\mu(\widetilde X_{t_{i-1}}+u)|\\
 &\qquad\qquad\leq c\, \bigl(|X_{t_{i-1}}-\widetilde X_{t_{i-1}}|+\sum_{j=1}^k 1_{\{(X_{t_{i-1}}+u-\xi_j)\, (\widetilde X_{t_{i-1}}+u-\xi_j)\leq 0\}}\bigr).
\end{aligned}
\end{equation}
Lemma~\ref{lab2} implies that there exists $c\in (0, \infty)$ such that 
\begin{equation}\label{b3}
\max_{i\in\{1,\dots,n\}}\,\EE[|X_{t_{i-1}}-\widetilde X_{t_{i-1}}|]\leq \frac{c}{n^{3/4}}.
\end{equation}
Moreover, for 
 all $i\in\{1, \ldots, n\}$, 
all $u\in\R$, all $j\in\{1, \ldots, k\}$ and all $\gamma\in(0, \infty)$ we have
\begin{equation}\label{f1}
\begin{aligned}
\EE[1_{\{(X_{t_{i-1}}+u-\xi_j)\, (\widetilde X_{t_{i-1}}+u-\xi_j)\leq 0\}}]&=\PP((X_{t_{i-1}}+u-\xi_j)\, (\widetilde X_{t_{i-1}}+u-\xi_j)\leq 0)\\
&\leq \PP(|X_{t_{i-1}}+u-\xi_j|\leq |X_{t_{i-1}}-\widetilde X_{t_{i-1}}|)\\
&\leq \PP(|X_{t_{i-1}}+u-\xi_j|\leq \gamma)+\PP(|X_{t_{i-1}}-\widetilde X_{t_{i-1}}|>\gamma).
\end{aligned}
\end{equation}
Due to the assumptions $n\in2\N$ and~\eqref{b5i} there exists 
$r\in\{1,\dots,n\}$ with $t_r= 1/2$.    Using Lemma \ref{Xpg} with $\tau =1/2$ we obtain that 
there exists $c\in(0, \infty)$ such that for
all $j\in\{1, \ldots, k\}$,
all $u\in\R$
 and all $\gamma\in(0, \infty)$,
\begin{equation}\label{f2}
\max_{i\in\{r+1,\dots,n\}}\,\PP(|X_{t_{i-1}}+u-\xi_j|\leq \gamma)\leq c\, \gamma.
\end{equation}
Furthermore, by 
Markov's inequality and Lemma \ref{lab2} there exists $c\in(0, \infty)$ such that for 
all $\gamma\in(0, \infty)$,
\begin{equation}\label{f3}
\max_{i\in\{1,\dots,n\}}\,\PP(|X_{t_{i-1}}-\widetilde X_{t_{i-1}}|>\gamma)\leq \max_{i\in\{1,\dots,n\}}\,\frac{\EE[|X_{t_{i-1}}-\widetilde X_{t_{i-1}}|^3]}{\gamma^3}\leq \frac{c}{\gamma^3\cdot n^{9/4}}.
\end{equation}
Choosing
\[
\gamma=n^{-\frac{9}{16}}
\]
we conclude from~\eqref{f1} to~\eqref{f3} that 
there exists $c\in(0, \infty)$ such that for 
all $u\in\R$ and all $j\in\{1, \ldots, k\}$, 
\begin{equation}\label{b4}
\max_{i\in\{r+1,\dots,n\}}\,\EE[1_{\{(X_{t_{i-1}}+u-\xi_j)\, (\widetilde X_{t_{i-1}}+u-\xi_j)\leq 0\}}]\leq c\, n^{-\frac{9}{16}}.
\end{equation}
Combining~\eqref{b1} to~\ \eqref{b3} and \eqref{b4} and observing \eqref{b5i} 
we conclude that 
there exist $c_1, c_2\in(0, \infty)$ such that 
for all $i\in\{r+1, \ldots, n\}$,
\begin{equation}\label{b9}
\EE[E_i^2]\leq c_1\, (t_i-t_{i-1})^2 \, \frac{1}{n^{9/16}}\leq \frac{c_2}{n^{5/2+1/16}}.
\end{equation}

Inserting the estimates~\eqref{uno} and~\eqref{b9} into~\eqref{nbv2} we conclude that there exists $c\in(0, \infty)$ such that for all $i\in\{r+1, \ldots, n\}$,
\[
\EE[A_i^2] \le 4\EE[C_i^2] + \frac{c}{n^{5/2+1/16}}
\]
which completes the proof of the lemma.
\end{proof}

\begin{lemma}\label{rest} Assume that $\mu$ satisfies ($\mu1$), ($\mu4$) and  ($\mu5$). Let $x_0\in\R$ and let $X,\widetilde X\colon [0,1]\times \Omega\to \R$ be strong solutions of the SDE~\eqref{generalsde} on the time-interval $[0,1]$ with initial value $x_0$ and driving Brownian motion $W$ and $\widetilde W$, respectively. Then 
there exists $c\in (0,\infty)$ such that for all $\ell\in\{1,\dots,k\}$ and all $i\in\{1,\dots,n\}$ with $t_{i}> 1/2$ it holds  
\begin{equation}\label{i3x}
\begin{aligned}
& \EE\Bigl[\Bigr(\int_{t_{i-1}}^{t_i} \bigl(\mu(X_{t_{i-1}} + W_s-W_{t_{i-1}})-\mu( X_{t_{i-1}} + \widetilde W_s-\widetilde W_{t_{i-1}})\bigr)\, ds\Bigr)^2\Bigr]\\
& \qquad\qquad\qquad \ge c\, (\mu(\xi_\ell+)-\mu(\xi_\ell-))^2\, (t_i-t_{i-1})^{5/2}.
\end{aligned}
\end{equation}
\end{lemma}

\begin{proof}
We have for all $i\in\{1,\dots,n\}$
 and all $s\in[t_{i-1}, t_i]$,
\begin{equation}\label{tmg33}
W_s-W_{t_{i-1}}=\tfrac{s-t_{i-1}}{t_i-t_{i-1}}\,(W_{t_i}-W_{t_{i-1}})+B_s, \quad \widetilde W_s-\widetilde W_{t_{i-1}}=\tfrac{s-t_{i-1}}{t_i-t_{i-1}}\,(W_{t_i}-W_{t_{i-1}})+\widetilde B_s.
\end{equation}
Hence, for all $i\in\{1,\dots,n\}$,
\begin{align*}
&\int_{t_{i-1}}^{t_i}(\mu(X_{t_{i-1}}+W_s- W_{t_{i-1}})-\mu(X_{t_{i-1}}+\widetilde W_s-\widetilde W_{t_{i-1}})) \, ds\\
&\qquad\qquad=\int_{0}^{t_i-t_{i-1}}(\mmu(X_{t_{i-1}}+\tfrac{s}{t_i-t_{i-1}}\,(W_{t_i}-W_{t_{i-1}})+B_{t_{i-1}+s}) \\
 & \qquad\qquad \qquad\qquad \qquad\qquad -\mmu(X_{t_{i-1}}+\tfrac{s}{t_i-t_{i-1}}\,(W_{t_i}-W_{t_{i-1}})+\widetilde B_{t_{i-1}+s})) \, ds.
\end{align*}

Recall from the proof of Lemma~\ref{mixed} that for all $i\in\{1,\dots,n\}$,  $(X_{t_{i-1}},\widetilde X_{t_{i-1}})$ and $(( W_s- W_{t_{i-1}},\widetilde W_s-\widetilde W_{t_{i-1}})_{s\in [t_{i-1},t_i]}$ are independent. Hence, for all $i\in\{1,\dots,n\}$, 
 $X_{t_{i-1}}$, $ W_{t_i}-W_{t_{i-1}}$, $(B_{t_{i-1}+s})_{s\in[0, t_i-t_{i-1}]}$, $(\widetilde B_{t_{i-1}+s})_{s\in[0, t_i-t_{i-1}]}$ are independent.
We may thus apply Lemma \ref{BrBr} to obtain  that for all $\ell\in\{1,\dots,k\}$ and for all $i\in\{1,\dots,n\}$,
\begin{equation}\label{tmg24}  
\begin{aligned}
& \EE\Bigl[\Bigr(\int_{t_{i-1}}^{t_i} \bigl(\mu(X_{t_{i-1}} + W_s-W_{t_{i-1}})-\mu( X_{t_{i-1}} + \widetilde W_s-\widetilde W_{t_{i-1}})\bigr)\, ds\Bigr)^2\Bigr]\\ & \qquad \qquad\geq \kappa\, (\mmu(\xi_\ell+)-\mmu(\xi_\ell-))^2 \, (t_i-t_{i-1})^2 \\
& \qquad\qquad\qquad\qquad \times \PP(X_{t_{i-1}}\in [\xi_\ell, \xi_\ell+\sqrt {t_i-t_{i-1}}])\, \PP(W_{t_i}-W_{t_{i-1}}\in [0, \sqrt {t_i-t_{i-1}}]),
\end{aligned}
\end{equation}
where
$\kappa$ is given by \eqref{kappa}.

Due to the assumptions $n\in2\N$ and~\eqref{b5i} there exists $r\in\{1,\dots,n\}$ with $t_{r}= 1/2$. Using Lemma~\ref{Xpg} with $\tau=1/2$, $M=\max_{\ell=1,\dots,k}|\xi_\ell|+1$ 
we see that 
there exists $c\in(0, \infty)$ such that for all $\ell\in\{1,\dots,k\}$ and   all   $i\in\{r+1, \ldots, n\}$,
\begin{equation}\label{tmg26}
\PP(X_{t_{i-1}}\in [\xi_\ell, \xi_\ell+\sqrt {t_i-t_{i-1}}])\geq c\, (t_i-t_{i-1})^{1/2}.
\end{equation}

   Combining~\eqref{tmg24} with~\eqref{h1} and~\eqref{tmg26} completes the proof of the lemma.
\end{proof}

We are ready to provide the appropriate  lower bound for the right hand side of~\eqref{gL1}.

\begin{lemma}\label{llbb}
Assume that $\mu$ satisfies ($\mu1$) to ($\mu5$). Let $x_0\in\R$ and let $X,\widetilde X\colon [0,1]\times \Omega\to \R$ be strong solutions of the SDE~\eqref{generalsde} on the time-interval $[0,1]$ with initial value $x_0$ and driving Brownian motion $W$ and $\widetilde W$, respectively.
Then 
there exist $c_1,c_2\in (0,\infty)$ such that 
\begin{equation}\label{aux11}
\EE[|X_1-\widetilde X_1|]\ge \frac{c_1}{n^{3/4}}\,(\max(0,(1-c_2/n^{1/16})))^{3/2}.
\end{equation}
\end{lemma}

\begin{proof} 
Put
\[
\Delta_i=\EE[|X_{t_i}-\widetilde X_{t_i}|^2]
\]
for $i\in\{0, \ldots, n\}$ and note
that 
\begin{equation}\label{LL5}
\EE[|X_1-\widetilde X_1|^2]=\Delta_n.
\end{equation}
Observing~\eqref{Brm} we see that  
for all $i\in\{1, \ldots, n\}$,
\begin{equation}\label{De}
\Delta_i=\EE\Bigl[\Bigl|X_{t_{i-1}}-\widetilde X_{t_{i-1}}+\int_{t_{i-1}}^{t_i}(\mu(X_s)-\mu(\widetilde X_s)) \, ds\Bigr|^2\Bigr]=\Delta_{i-1} +2m_i+d_i,
\end{equation}
where 
\[
m_i=\EE\Bigl[(X_{t_{i-1}}-\widetilde X_{t_{i-1}})\,\int_{t_{i-1}}^{t_i}(\mu(X_s)-\mu(\widetilde X_s)) \, ds\Bigr] \,\text{ and }\, d_i=\EE\Bigl[\Bigl|\int_{t_{i-1}}^{t_i}(\mu(X_s)-\mu(\widetilde X_s)) \, ds\Bigr|^2\Bigr].
\]

Using Lemma~\ref{mixed} we have for all $i\in\{1,\dots,n\}$ that
$m_i\ge 0$. By~\eqref{b5i} there exists $r\in\{0, \ldots, n\}$ such that $t_r=1/2$. Combining Lemma~\ref{diagonal1} and Lemma~\ref{rest} and observing property ($\mu3$) we conclude that there exist $c_1,c_2\in (0,\infty)$ such that for $i\in\{r+1,\dots,n\}$, 
\[
d_i \ge c_1\, (t_i-t_{i-1})^{5/2}  - c_2\,\frac{1}{n^{5/2+1/16}}.
\]
Hence, observing \eqref{oo33} we obtain that there exists $c_3\in(0, \infty)$ such that
\begin{equation}\label{uzt1}
\begin{aligned}
\EE[|X_1-\widetilde X_1|^2] &= 2\sum_{i=1}^n m_i+\sum_{i=1}^n d_i \ge \sum_{i=1}^n d_i\ge \sum_{i=r+1}^n d_i \\
& \ge c_1\sum_{i=r+1}^n(t_i-t_{i-1})^{5/2}  - \frac{c_2}{n^{3/2+1/16}}\ge \frac{c_3}{n^{3/2}} - \frac{c_2}{n^{3/2+1/16}}\\
&= \frac{c_3}{n^{3/2}}\cdot \Bigl(1-\frac{c_2}{c_3}\cdot \frac{1}{n^{1/16}}\Bigr).
\end{aligned}
\end{equation}

Employing  Lemma~\ref{lab2} with $p=4$ we may proceed similar to the end of the proof of Theorem~\ref{Thm2}, see~\eqref{star1} and~\eqref{star2}, to conclude with the help of 
H\"older's inequality that 
 there exist $c_1,c_2,c_3\in (0,\infty)$ such that
\[
\frac{c_1}{n^{3/2}}\cdot \max(0,(1-c_2/n^{1/16}))\le \EE[|X_1-\widetilde X_1|^2]\le \EE[|X_1-\widetilde X_1|]^{2/3}\,\frac{c_3}{n},
\] 
which implies that  
there exist $c_1,c_2\in (0,\infty)$ such that
\begin{equation}\label{aux12}
\EE[|X_1-\widetilde X_1|] \ge \frac{c_1}{n^{3/4 }}\,(\max(0,(1-c_2/n^{1/16})))^{3/2}
\end{equation}
and hereby finishes the proof of the lemma.
\end{proof}

Combining Lemma~\ref{lemmanew02} with $p=1$ and Lemma~\ref{llbb} yields Proposition~\ref{unprop} and hereby finishes the proof of Theorem~\ref{Thm1}.

\section*{Appendix}\label{app}

It is well-known that the components of a bivariate 
normal random variable $(Z,Y)$ with  $\text{Cov}(Z,Y)\ge 0$ are positively associated, whence, in particular, $\text{Cov}(f(Z),g(Y))\ge 0$ holds for all 
increasing 
$f,g\colon\R\to\R$ such that $\text{Cov}(f(Z),g(Y))$
 exists. See, e.g.~\cite[Theorem 5.1.1]{To90}. The following lemma strengthens this result in the case when $f$ and $g$ are piecewise Lipschitz continuous.

\begin{lemma}\label{TongIn}
Let $\rho\in[0,1]$ and $
(Z,Y)\sim N\Bigl(0, \Bigl({\small 
\begin{array}{rr}1 & \rho \\ \rho & 1 \\\end{array}}\Bigr)\Bigr)
$.  Moreover, let $k,l\in\N$ and $-\infty=a_0<a_1<\ldots<a_k<a_{k+1}=\infty$ and $-\infty=b_0<b_1<\ldots<b_l<b_{l+1}=\infty$, and let $f,g\colon\R\to\R$ satisfy 
\begin{itemize}
\item[(i)] $f,g$ are both increasing or both decreasing,
\item[(ii)] $f$ is Lipschitz continuous on the interval $(a_{i-1}, a_i)$  for all $i\in\{1, \ldots, k+1\}$ and $g$ is Lipschitz 
continuous on the interval $(b_{j-1}, b_j)$  for all $j\in\{1, \ldots, l+1\}$.
\end{itemize}
Then it holds
\begin{equation}\label{k2}
\begin{aligned}
&\EE[f(Z)\,g(Y)]-\EE[f(Z)]\, \EE[g(Y)]\\
&\quad \geq\sum_{i=1}^{k} \sum_{j=1}^{l} (f(a_i+)-f(a_i-))\, (g(b_j+)-g(b_j-)) \,  \frac{1}{2\pi}\,e^{-\frac{a_i^2}{2}}\, \int_0^{\rho} \frac{1}{\sqrt{1-u^2}}\, e^{-\frac{(b_j-a_iu)^2}{2 (1-u^2)}} du.
\end{aligned}
\end{equation}
\end{lemma}
\begin{proof} 
Clearly, we may assume that $f$ and $g$ are both increasing.
Employing the assumption (ii) and Lemma~\ref{basics} we see that there exists $c\in (0,\infty)$ such that for all $x\in\R$,
\begin{equation}\label{z1}
|f(x)| + |g(x)| \le c\cdot (1+|x|),
\end{equation}
and that all of the limits $f(a_i-), f(a_i+), g(b_j-), g(b_j+)$, $i\in\{1,\dots,k\},j\in\{1,\dots,l\}$ exist and are finite. 
Hence all of the 
expected values
on the left hand side of~\eqref{k2} are well-defined and finite, and the right hand side of~\eqref{k2} is well-defined as well.

 Clearly, \eqref{k2} holds if $\rho=0$. 
Next, assume that $\rho\in(0,1)$. We proceed similar to the proof of~\cite[Theorem 5.1.1]{To90}.
For $u\in[0, \rho]$ let 
\[
(Z_u,Y_u)\sim N\Bigl(0, \Bigl({\small \begin{array}{rr}1 & u  \\ u & 1  \\\end{array}}\Bigr)\Bigr)
\]
and define $\psi\colon[0, \rho]\to[0, \infty)$ by
\[
\psi(u)=\EE[f(Z_u)\, g(Y_u)].
\]
Then
\begin{equation}\label{d11}
\EE[f(Z)\, g(Y)]-\EE[f(Z)]\, \EE[g(Y)]=\psi(\rho)-\psi(0).
\end{equation}

Using the well-known fact that 
for all  $u\in [0,\rho]$
the conditional distribution of $Y_u$ given $Z_u$ satisfies 
$\PP^{Y_u|Z_u=z} = N(zu,1-u^2)$ for $\PP^{Z_u}$-almost all $z\in\R$, 
    we obtain for all $u\in[0, \rho]$ that
\[
\psi(u)=\int_{\R}\int_{\R} f(z)\, g(y)\, h(z,y,u)\,\varphi(z)\, dy\, dz,
\]
where the functions $h\colon\R^2\times [0, \rho]\to\R$ and $\varphi\colon\R\to\R$ are defined by
\[
h(z,y,u)= \frac{1}{\sqrt{2\pi(1-u^2)}}\,e^{-\tfrac{(y-zu)^2}{2(1-u^2)}},\quad \varphi(z)=\frac{1}{\sqrt{2\pi}}\,e^{-\tfrac{z^2}{2}}.
\]
Let $z,y\in\R$. For all $u\in[0, \rho]$,
\[
\frac{\partial}{\partial u}h(z,y,u)=\Bigl(\frac{u}{1-u^2}-\frac{(y-zu)(yu-z)}{(1-u^2)^2}\Bigr)\, h(z,y,u).
\]
Since
\[
\sup_{u\in[0, \rho]} e^{-\tfrac{(y-zu)^2}{2(1-u^2)}}\leq 
\sup_{u\in[0, \rho]}
e^{-\tfrac{(y-zu)^2}{2}}\leq e^{-\tfrac{y^2}{2}+|yz|\rho}\leq e^{-\tfrac{y^2(1-\rho)}{2}+\tfrac{z^2\rho}{2}}
\]
we obtain that
\[
\sup_{u\in[0, \rho]}\Bigl|\frac{\partial}{\partial u}h(z,y,u)\Bigr|\leq \Bigl(\frac{1}{1-\rho^2}+\frac{(|y|+|z|)^2}{(1-\rho^2)^2}\Bigr)\,  \frac{1}{\sqrt{2\pi(1-\rho^2)}}\, e^{-\tfrac{y^2(1-\rho)}{2}+\tfrac{z^2\rho}{2}}.
\]
By the latter fact and~\eqref{z1} we may 
apply the dominated convergence theorem to conclude that $\psi$ is continuously differentiable with
\begin{equation}\label{n3}
\psi'(u)=\int_{\R}\int_{\R} f(z)\, g(y)\, \frac{\partial}{\partial u}h(z,y,u)\,\varphi(z)\, dy\, dz, \quad u\in[0, \rho].
\end{equation}
Below we show that for all $u\in(0, \rho]$,
\begin{equation}\label{n1}
\psi'(u)\geq \sum_{i=1}^{k} \sum_{j=1}^{l} (f(a_i+)-f(a_i-))\, (g(b_j+)-g(b_j-)) \, \varphi(a_i) \, h(a_i,b_j,u).
\end{equation}
The latter estimate, the equality 
\[
\psi(\rho)-\psi(0)=\int_0^\rho \psi'(u)\,du
\]
and \eqref{d11} imply~\eqref{k2}.

It remains to prove \eqref{n1}. Straightforward calculations show that for all $z,y\in\R$ and all $u\in[0, \rho]$,
\[
\frac{\partial}{\partial z}h(z,y,u)=\frac{u(y-zu)}{1-u^2}\, h(z,y,u), \qquad \frac{\partial^2}{\partial z^2}h(z,y,u)=\Bigl(\frac{u^2(y-zu)^2}{(1-u^2)^2}-\frac{u^2}{1-u^2}\Bigr)\, h(z,y,u),
\]
and therefore for all $z,y\in\R$ and all $u\in(0, \rho]$,
\[
\frac{\partial}{\partial u}h(z,y,u)=-\frac{1}{u}\,\Bigl(\frac{\partial^2}{\partial z^2}h(z,y,u)-z \frac{\partial}{\partial z}h(z,y,u)\Bigr).
\]
Observing \eqref{n3} we may thus conclude that for all $u\in(0, \rho]$,
\begin{equation}\label{x}
\psi'(u)=-\frac{1}{u}\,\int_{\R} g(y)\int_{\R} f(z)\,\varphi(z)\, \Bigl(\frac{\partial^2}{\partial z^2}h(z,y,u)-z \frac{\partial}{\partial z}h(z,y,u)\Bigr)\, dz\, dy.
\end{equation}

Let $f_0\colon (-\infty, a_1]\to\R$, $f_1\colon [a_{1}, a_2]\to\R, \ldots, f_{k}\colon [a_k, \infty)\to\R$  denote
 the continuous extensions of $f_{|(-\infty, a_1)}$, $f_{|(a_{1}, a_2)}, \ldots,f_{|(a_k, \infty)}$ on $(-\infty, a_1]$, $[a_{1}, a_2], \ldots, [a_k, \infty)$, respectively.  
The assumption (ii) implies that for every $i\in\{0, \ldots, k\}$ the function $f_i$ is Lipschitz continuous on its domain and therefore has a Lebesgue density $f_i'$.
Applying the integration by parts formula and observing that $\varphi'(z)=-z \varphi(z)$ for all $z\in\R$ 
as well as
\[
\lim_{z\to -\infty} f(z)\,\varphi(z)\,\frac{\partial}{\partial z}h(z,y,u) = \lim_{z\to \infty} f(z)\,\varphi(z)\,\frac{\partial}{\partial z}h(z,y,u) = 0
\]
for all $y\in\R$ and all $u\in [0,\rho]$,
we therefore obtain that for all $y\in\R$ and all $u\in(0, \rho]$,
\begin{equation}\label{xx}
\begin{aligned}
&\int_{\R} f(z)\,\varphi(z)\, \frac{\partial^2}{\partial z^2}h(z,y,u)\, dz\\
&\qquad\quad\quad=\sum_{i=0}^{k}\int_{a_i}^{a_{i+1}} f_i(z)\,\varphi(z)\, \frac{\partial^2}{\partial z^2}h(z,y,u)\, dz\\
&\qquad\quad\quad= \sum_{i=1}^{k}(f(a_i-)-f(a_i+))\, \varphi(a_i)\, \frac{\partial}{\partial z} h(a_i,y,u)
-\sum_{i=0}^{k}\int_{a_i}^{a_{i+1}} (f_i\cdot\varphi)'(z)\,  \frac{\partial}{\partial z}h(z,y,u)\, dz\\
&\qquad\quad\quad=\sum_{i=1}^{k}(f(a_i-)-f(a_i+))\, \varphi(a_i)\, \frac{\partial}{\partial z} h(a_i,y,u)
-\sum_{i=0}^{k}\int_{a_i}^{a_{i+1}} f_i'(z)\,\varphi(z)\,  \frac{\partial}{\partial z}h(z,y,u)\, dz\\
&\qquad\quad\qquad\quad\quad +\int_{\R} f(z)\,\varphi(z)\,  z\frac{\partial}{\partial z}h(z,y,u)\, dz.
\end{aligned}
\end{equation}
Using~\eqref{x} and~\eqref{xx} we see that for  all $u\in(0, \rho]$,
\begin{equation}\label{n8}
\psi'(u)=v(u)+w(u),
\end{equation}
where the functions $v,w\colon (0, \rho]\to\R$ are defined by
\begin{align*}
v(u)&=\sum_{i=1}^{k}(f(a_i+)-f(a_i-))\, \varphi(a_i)\, \frac{1}{u}\int_{\R} g(y)\, \frac{\partial}{\partial z} h(a_i,y,u)\, dy,\\
w(u)&= \sum_{i=0}^{k}\int_{a_i}^{a_{i+1}}f_i'(z)\,\varphi(z)\, \Bigl(\frac{1}{u}\int_{\R} g(y) \frac{\partial}{\partial z}h(z,y,u) \, dy\Bigr) \, dz.
\end{align*}
Note that for all $z,y\in\R$ and $u\in[0, \rho]$,
\begin{equation}\label{z2}
\frac{\partial}{\partial z}h(z,y,u)=-u\, \frac{\partial}{\partial y} h(z,y,u).
\end{equation}
Let $g_0\colon (-\infty, b_1]\to\R$, $g_1\colon [b_{1}, b_2]\to\R, \ldots, g_{l}\colon [b_l, \infty)\to\R$  denote
 the continuous extensions of $g_{|(-\infty, b_1)}$, $g_{|(b_{1}, b_2)}, \ldots,g_{|(b_l, \infty)}$ on $(-\infty, b_1]$, $[b_{1}, b_2], \ldots, [b_l, \infty)$, respectively. 
The assumption (ii) implies that for every $j\in\{0, \ldots, l\}$ the funtion $g_j$ is Lipschitz continuous on its domain and therefore has a Lebesgue density $g_j'$.
   Using~\eqref{z2}, applying the integration by parts formula 
and observing that
\[
\lim_{y\to -\infty} g(y)\,h(z,y,u) = \lim_{y\to \infty} g(y)\,h(z,y,u) = 0
\]
for all $z\in\R$ and all $u\in [0,\rho]$,
we therefore obtain that for all $z\in\R$ and all $u\in(0, \rho]$,
\begin{align*}
\frac{1}{u}\int_{\R} g(y)\, \frac{\partial}{\partial z} h(z,y,u)\, dy&=- \int_{\R} g(y)\, \frac{\partial}{\partial y} h(z,y,u)\, dy=- \sum_{j=0}^{l} \int_{b_j}^{b_{j+1}} g_j(y)\, \frac{\partial}{\partial y} h(z,y,u)\, dy\\
&=\sum_{j=1}^{l} (g(b_j+)-g(b_j-))\, h(z,b_j,u)+\sum_{j=0}^{l} \int_{b_j}^{b_{j+1}} g_j'(y)\, h(z,y,u)\, dy.
\end{align*}
Since $f$ and $g$ are both increasing
we may assume that $f_i',g_j'\geq 0$ for all $i\in\{0, \ldots, k\}$ and $j\in\{0,\dots,l\}$ and  we conclude 
that for all $u\in(0, \rho]$,
\[
v(u)\geq \sum_{i=1}^{k} \sum_{j=1}^{l} (f(a_i+)-f(a_i-))\, (g(b_j+)-g(b_j-)) \, \varphi(a_i) \, h(a_i,b_j,u)
\]
and
\begin{align*}
w(u)&\geq \sum_{j=1}^{l} (g(b_j+)-g(b_j-)) \sum_{i=0}^{k}\int_{a_i}^{a_{i+1}}f_i'(z)\,\varphi(z) \, h(z,b_j,u)\, dz\geq 0.
\end{align*}
The latter two estimates together with \eqref{n8} yield \eqref{n1} and complete the proof of the lemma in the case $\rho\in(0,1)$.

Finally, assume that $\rho=1$. Then $Z=Y$ $\PP$-a.s. Let $U\sim N(0, 1)$ be independent of $Z$ and for $s\in[0,1)$ put
\[
V_s= s\, Z+\sqrt{1-s^2} \, U.
\]
 Observe that \[
(Z,V_s)\sim N\Bigl(0, \Bigl({\small \begin{array}{rr}1 & s  \\ s  & 1  \\\end{array}}\Bigr)\Bigr)
\]  for all $s\in[0,1)$ 
and that $g$ has at most finitely many discontinuity points.
Hence, $\PP$-a.s.,
 \[
 \lim_{s\to 1}g(V_s)=g(Y). 
 \]
Observing~\eqref{z1} we may thus 
apply the dominated convergence theorem to conclude
\[
\EE[f(Z)\, g(Y)]-\EE[f(Z)]\, \EE[g(Y)]= \lim_{s\to 1} \bigl(\EE[f(Z)\, g(V_s)]-\EE[f(Z)]\, \EE[g(V_s)]\bigr).
\]
Applying \eqref{k2} with $Y=V_s$ for $s\in[0,1)$ and using the fact that for all $a,b\in\R$,                   
\begin{align*}
\lim_{s\to 1} \int_0^{s} \frac{1}{\sqrt{1-u^2}}\, e^{-\frac{(b-au)^2}{2 (1-u^2)}} du=\int_0^{1} \frac{1}{\sqrt{1-u^2}}\, e^{-\frac{(b-au)^2}{2 (1-u^2)}} du
\end{align*}
 finishes the proof of \eqref{k2} in the case $\rho=1$ and completes the proof of the lemma.
\end{proof}

\bibliographystyle{acm}
\bibliography{bibfile}

\end{document}